\documentclass[12pt]{amsart}

\usepackage{tikz-cd}

\usepackage{amssymb}
\usepackage{amsthm}
\usepackage{amscd}

\usepackage[mathscr]{eucal}
\usepackage{enumitem}

\usepackage{hyperref}

\usepackage{color}
\newtheorem{Theorem}{Theorem}[section]
\newtheorem{theorem}[Theorem]{Theorem}
\newtheorem{Proposition}[Theorem]{Proposition}
\newtheorem{proposition}[Theorem]{Proposition}
\newtheorem{Corollary}[Theorem]{Corollary}
\newtheorem{corollary}[Theorem]{Corollary}
\newtheorem{Lemma}[Theorem]{Lemma}
\newtheorem{lemma}[Theorem]{Lemma}
\newtheorem{Fact}[Theorem]{Fact}
\newtheorem{fact}[Theorem]{Fact}
\newtheorem{remark/def}[Theorem]{Remark/Definition}

\newtheorem{Claim}[Theorem]{Claim}
\newtheorem{claim}[Theorem]{Claim}

\theoremstyle{definition}

\newtheorem{example}[Theorem]{Example}
\newtheorem{Remark}[Theorem]{Remark}
\newtheorem{remark}[Theorem]{Remark}
\newtheorem{Definition}[Theorem]{Definition}
\newtheorem{definition}[Theorem]{Definition}

\newtheorem{notation}[Theorem]{Notation}

\newtheorem{note}[Theorem]{Note}
\newtheorem{Question}[Theorem]{Question}
\newtheorem{question}[Theorem]{Question}
\newtheorem{def/rem}[Theorem]{Definition/Remark}
\newtheorem{def/fact}[Theorem]{Definition/Fact}
\newtheorem{not/rem}[Theorem]{Notation/Remark}
\newsavebox{\indbin}
\savebox{\indbin}{\begin{picture}(0,0)
\newlength{\gnu}
\settowidth{\gnu}{$\smile$} \setlength{\unitlength}{.5\gnu}
\put(-1,-.65){$\smile$} \put(-.25,.1){$|$}
\end{picture}}
\def \indo {\mathop{\smile \hskip -0.9em ^| \ }}

\newcommand{\ind}[3]
{\mbox{$\begin{array}{ccc}
\mbox{$#1$} & \usebox{\indbin} & \mbox{$#2$} \\
        & \mbox{$#3$} &
\end{array}$}}

\newcommand{\be}{\begin{enumerate}}
\newcommand{\bi}{\begin{itemize}}

\newcommand{\bd}{\begin{defn}}

\newcommand{\bt}{\begin{theorem}}
\newcommand{\bl}{\begin{lemma}}
\newcommand{\ee}{\end{enumerate}}
\newcommand{\ei}{\end{itemize}}
\newcommand{\ed}{\end{defn}}
\newcommand{\et}{\end{theorem}}
\newcommand{\el}{\end{lemma}}

\newcommand{\la}{\langle}
\newcommand{\ra}{\rangle}

\newcommand{\ov}{\overline}

\newcommand{\CP}{{\mathcal P}}

\newcommand{\CR}{{\mathcal R}}
\newcommand{\CA}{{\mathcal A}}

\newcommand{\CC}{{\mathcal C}}
\newcommand{\CE}{{\mathcal E}}

\newcommand{\CM}{{\mathcal M}}
\newcommand{\CN}{{\mathcal N}}

\newcommand{\BQ}{{\mathbb Q}}
\newcommand{\BR}{{\mathbb R}}

\newcommand{\dom}{\mbox{dom}}

\newcommand{\id}{\operatorname{id}}
\newcommand{\Aut}{\operatorname{Aut}}
\newcommand{\aut}{\operatorname{Aut}}
\newcommand{\Autf}{\operatorname{Autf}}
\newcommand{\autf}{\operatorname{Autf}}

\newcommand{\LL}{\operatorname{L}}
\newcommand{\gal}{\operatorname{Gal}}
\newcommand{\gall}{\gal_{\LL}}
\newcommand{\KP}{\operatorname{KP}}

\newcommand{\galkp}{\gal_{\KP}}

\newcommand\floor[1]{\lfloor#1\rfloor}

\newcommand{\Th}{\operatorname{Th}}

\def\fix{\operatorname{fix}}
\def\res{\operatorname{res}}

\def\eq{\operatorname{eq}}

\def\cl{\operatorname{cl}}

\def\dcl{\operatorname{dcl}}
\def\dim{\operatorname{dim}}
\def\dom{\operatorname{dom}}

\def\acl{\operatorname{acl}}

\def\ker{\operatorname{Ker}}

\def\Im{\operatorname{Im}}

\def\tp{\operatorname{tp}}

\def\stp{\operatorname{stp}}
\def\ltp{\operatorname{Ltp}}

\def\Lstp{\operatorname{Lstp}}

\def\a{\bar{a}}
\def\b{\bar{b}}

\def\p{\bar{p}}

\def\raw{\rightarrow}

\def\mor{\operatorname{Mor}}
\def\supp{\operatorname{supp}}
\def\Bd{\partial}

\title[The relativized Lascar groups ]{The relativized Lascar groups, type-amalgamation, and algebraicity}

\author[J.~Dobrowolski, B.~Kim, A.~Kolesnikov, and J.~Lee]{Jan Dobrowolski, Byunghan Kim, Alexei Kolesnikov, and Junguk Lee}
\address{School of Mathematics\\ University of Leeds\\
Leeds, United Kingdom}
\address{Department of Mathematics\\ Yonsei University\\
Seoul, Korea}
\address{Department of Mathematics\\ Towson University\\  MD, USA}
\address{Department of Mathematics\\ Yonsei University\\
Seoul, Korea}

\email{J.Dobrowolski@leeds.ac.uk}
\email{bkim@yonsei.ac.kr}
\email{AKolesnikov@towson.edu}
\email{ljw@yonsei.ac.kr}

\thanks{The first author was supported by European Union's Horizon 2020 research
and innovation programme under the Marie Sklodowska-Curie grant agreement No 705410. The second and fourth authors were supported by Samsung Science Technology Foundation under Project Number
SSTF-BA1301-03.}
\begin{document}

\begin{abstract}  
In this paper we study the relativized Lascar Galois group of a strong type.
The group is a quasi-compact connected topological group, and if in addition the underlying theory $T$ is $G$-compact, then the group is 
compact. We apply compact group theory to obtain model theoretic results in this note. 

For example, we use the divisibility of the Lascar group of a strong type to show that, in a simple theory, such types have a certain model theoretic property that we call divisible amalgamation. 

The main result of this paper is that if $c$ is a finite tuple algebraic over a tuple $a$, the Lascar group of $\stp(ac)$ is abelian, and the underlying theory is $G$-compact, then the   Lascar groups of $\stp(ac)$ and of $\stp(a)$ are isomorphic. 
To show this, we prove a purely compact group-theoretic result that any compact connected abelian 
group is isomorphic to its  quotient by every finite subgroup.  

Several (counter)examples arising in connection with the theoretical development of this note are presented as well. For example, we show that, in the main result above, neither the assumption that the Lascar group of $\stp(ac)$ is abelian, nor the assumption of $G$-compactness can be removed. 
\end{abstract}

\maketitle

%\section{Introduction}

%We begin to review some notions and facts from \cite{GKK1},\cite{GKK} and \cite{KKL} in the more general context of an abstract independence relation in any theory.

Given a complete theory $T$, the notion of the Lascar (Galois) group $\gall(T)$ was introduced by D.~Lascar (see \cite{CLPZ}). The Lascar group only depends on the theory and it is a quasi-compact topological group with respect 
to a quotient  topology of  a certain  Stone  type space  over a model  (\cite{CLPZ} or  \cite{Z}).  
More recently, the notions of the {\em relativized} Lascar groups were introduced in \cite{DKL}
(and studied also in \cite{KNS} in the context of topological dynamics).  Namely, given a type-definable set $X$ in a large saturated model of the theory $T$,  we consider the group of automorphisms restricted to the set $X$ quotiented by the group of restricted automorphisms fixing the Lascar types of the sequences from $X$ of length $\lambda$. The relativized Lascar groups also only depend on $T$, on the type that defines $X$, and on the cardinal $\lambda$. These groups are endowed with the quasi-compact quotient topologies induced by the canonical 
surjective maps from $\gall(T)$ to the relativized Lascar groups. 

%If we assume the empty set is algebraically closed (in its eq) then all the mentioned (relativized) Lascar groups  are connected. If we further assume $T$ is $G$-compact over $\emptyset$ then they are compact (so Hausdorff).
 
This paper continues the study started in \cite{DKL}, where a connection between the relativized Lascar
groups of a strong type and the first homology group of the strong type  was established. If $T$ is $G$-compact (for example when $T$ is simple) then the relativized  Lascar group of a strong type is compact and connected, and we use compact group theory to obtain results presented here. 

A common theme for the results in this paper is the connection between group theoretic properties of the Lascar group of a strong type and the model theoretic properties of the type.
In Section 1, we use the fact that any compact connected group is divisible to show that any strong type $p$ in a simple theory has a property we call divisible amalgamation. The property would follow from the Independence theorem if $p$ was a Lascar type, but here $p$ is only assumed to be a strong type. 
Moreover, we give amalgamation criteria for when the Lascar group of the strong type of a model is abelian, and for when the strong type of a model is a Lascar type. 

In Section 2, we study morphisms between the relativized Lascar groups. The main goal is to understand the connection between the Lascar group of $\stp(a)$, for some tuple $a$, and the Lascar group of $\stp(\acl(a))$ or $\stp(ac)$, for a finite tuple $c\in \acl(a)$. We  prove that  if $T$ is $G$-compact, then for a finite tuple  $c$ algebraic over $a$, the restriction map from a Lascar group of $\stp(ac)$ to that of $\stp(a)$ is a covering map. In addition, if the Lascar group of $\stp(ac)$ is abelian  then this group is isomorphic to
that of $\stp(a)$ as topological groups. In order to achieve this,  we separately prove a purely compact group theoretical result that any compact connected abelian group is isomorphic to its quotient by a finite subgroup. We also give an example showing that the abelianness of the relativized Lascar group is essential in the isomorphism result. 

In Section 3, we mainly present 3 counterexamples: a non $G$-compact theory where $\stp(a)$ is a Lascar type but $\stp(\acl(a))$ is not a Lascar type; an example showing that in the mentioned isomorphism  result in Section 2, the tuple $c$ being finite is essential; and an example answering a question raised in \cite{KL}, namely 
a Lie group structure example where an RN-pattern minimal 2-chain is not equivalent to a Lascar pattern 2-chain having the same boundary.    
 
\medskip 
In the remaining part of this section we recall  the definitions  and terminology of basic notions, which, unless said otherwise, we  use
throughout this note. We work in a large saturated model $\CM(=\CM^{\eq})$ of a complete theory
$T$, and we use the standard notation. So  $A, B,\dots$ and $M,N,\dots$ are small subsets and elementary submodels of $\CM$, respectively. Lower-case letters $a,b,\dots$ will denote 
tuples of elements from $\CM$, possibly infinite. We will explicitly specify when a tuple is assumed to be finite. 

To simplify the notation, we state the results for types over the empty set (rather than over $\acl(\emptyset)$).
This does not reduce the generality because, after naming a parameter set, we may assume that $\dcl(\emptyset)=\acl(\emptyset)$. We  fix a complete {\em strong} type $p(x)\in S(\emptyset)$ with possibly infinite arity of $x$. 
For tuples $a,b$, we write $a\equiv_Ab$ ($a\equiv^s_Ab$,  resp.) to mean that they have the same (strong,  resp.) type over $A$. Note that $a\equiv^s_Ab$ if and only if $a\equiv_{\acl(A)}b$.

Let us recall the definitions of the Lascar groups and types. These are well-known notions in model theory. In particular, the Lascar group depends only on $T$ and is
a {\em quasi-compact}  group under the topology introduced in  \cite{CLPZ} or \cite{Z}. 
Moreover, due to our assumption 
$\dcl(\emptyset)=\acl(\emptyset)$,  the group  is {\em connected} as well. Recall that a topological space is {\em compact} if it is quasi-compact and Hausdorff.  

\begin{Definition}
\bi
\item
$\autf(\CM)$ is the normal  subgroup of $\aut(\CM)$ generated by 
$$\{f \in \aut(\CM) \mid f \mbox{ pointwise fixes some model } M \prec \CM\}.$$ 

\item For tuples $a,b \in \CM$, we say they have {\em the same Lascar  type}, written $a\equiv^Lb$ or $\ltp(a)=\ltp(b)$,
 if there is $f\in\autf(\CM) $ such that 
$f(a)=b$. We say $a,b$ have the same KP(Kim--Pillay)-type (write $a\equiv^{KP} b$) if they are in the same
class  of any bounded $\emptyset$-type-definable equivalence relation.  

\item  The group $\gall(T):=\aut(\CM)/\autf(\CM)$ is called the {\em Lascar   (Galois)  group} of $T$.

\item We say $T$ is {\em $G$-compact} if $\gall(T)$ is compact, equivalently $\{\mbox{id}\}$ is closed in $\gall(T)$. 
\ei
\end{Definition}

Note that the above  is the definition of ``$G$-compactness over $\emptyset$'', but for convenience throughout this paper we omit ``over $\emptyset$.''

\begin{Fact}
For tuples $a,b$, we have $a\equiv^L b$ if and only if the Lascar distance between $a$, $b$ is finite, i.e., there are finitely many indiscernible sequences 
$I_1,\ldots, I_n$, and tuples $a=a_0, a_1,\ldots, a_n=b$  such that  each
of $a_{i-1}I_i$ and $a_iI_i$ is an indiscernible sequence for $i=1,\ldots,n$.
\end{Fact}

Now let us recall, mainly from  \cite{DKL}, the definitions of various relativized Lascar groups of $p$ and related facts.

\begin{Definition}
\bi
\item
$\aut(p):=\{ f \restriction p(\CM):\  f\in \aut(\CM)\}$; 
\item for a cardinal $\lambda> 0$, $\autf^{\lambda}(p)=\autf_{\fix}^{\lambda}(p):=$
	$$\{\sigma\in \aut(p)\mid  
	\mbox{ for any } \a=(a_i)_{i< \lambda} \mbox{ with }  a_i\models p,
	\  \a\equiv^L \sigma(\a)
	 \};$$ 
\item
	$\autf_{\fix}(p):=$
	$$\{\sigma\in \aut(p) \mid  \a\equiv^L \sigma(\a)\mbox{ where }	\a \mbox{ is some enumeration of }
	p(\CM) \};$$
	and $\autf_{\res}(p):=\{ f \restriction p(\CM):\  f\in \autf(\CM)\}$.
	
\ei
\noindent Notice that $\autf^{\lambda}(p)$, $\autf_{\fix}(p)$, and  $\autf_{\res}(p)$ are normal subgroups of $\aut(p)$. 

\bi	 	
	%\item $\gall^{\res}(\Sigma(\CM)):=\aut(\Sigma(\CM))/\autf_{\res}(\Sigma(\CM))$;
	\item $\gall^{\lambda}(p)=\gall^{\fix, \lambda}(p):=\aut(p)/\autf^{\lambda}(p)$; \footnote{Similarly, $\autf^{\lambda}_{\KP}(p)$ is defined as the group of automorphisms in $\aut(p)$ fixing the KP-type of any $\lambda$-many realizations of $p$,
and $\galkp^{\lambda}(p):=\aut(p)/\autf^{\lambda}_{\KP}(p)$. Then in this paper, one may work with $\galkp^{\lambda}(p)$ instead of $\gall^{\lambda}(p)$, and remove the assumption of $T$ being $G$-compact where that is assumed.} 
	\item $\gall^{\fix}(p):=\aut(p)/\autf_{\fix}(p)$,
and  $\gall^{\res}(p):=\aut(p)/\autf_{\res}(p)$.
	\ei
\end{Definition}
%Consider
%$\nu ':S_M(N)\to \gall^{\res}(\Sigma(\CM))$ such that $\nu '\mu:\aut(\CM)\to \gall^{\res}(\Sigma(\CM))$ is the quotient map
 %sending $f$ to $(f\restriction \Sigma(\CM))\autf_{\res}(\Sigma(\CM))$.  We put on $\gall^{\res}(\Sigma(\CM))$ the quotient topology with
%respect to $\nu '$. In a similar manner, by considering the map $\nu'':S_M(N)\to \gall^{\fix}(\Sigma(\CM))$ such that
%$\nu''\mu:\aut(\CM)\to \gall^{\fix}(\Sigma(\CM))$ is the quotient map, we equip the group $\gall^{\fix}(\Sigma(\CM))$ with the topology induced
%by $\nu''$, which coincides with the topology induced on $\gall^{\fix}(\Sigma(\CM))$ by the quotient
%map $\gall(T)\to \gall^{\fix}(\Sigma(\CM))$. Both of these topologies are quotients of the standard topology on the (global) Lascar-Galois group (defined in \cite{Z},
%and hence, equipped with them, $\gall^{\res}(\Sigma(\CM))$ and $\gall^{\fix}(\Sigma(\CM))$ become topological groups.

 We will give an example (in Example \ref{g1neg2}) where $\gall^1(p)$ and  $\gall^2(p)$ are distinct.
In \cite[Remark 3.4]{DKL}, a canonical topology on each of the above groups was defined. With these topologies,
they become quotients of the  topological group $\gall(T)$.

\begin{Fact} \cite{DKL}  \label{rel.lascargp}
$\autf^{\omega}(p)=\autf_{\fix}(p)$, and, for each $\lambda(\leq \omega)$, 
$\gall^{\lambda}(p)$ does not depend on the choice of a monster model,
and is a quasi-compact connected 
topological group. Hence, if $T$ is $G$-compact, $\gall^{\lambda}(p)$ is a
compact  connected group.

If $p(x)$ is a type of a model,
then $\autf^1(p)=\autf_{\fix}(p)=\autf_{\res}(p)$, $\gall^1(p)=\gall^{\fix}(p)=\gall^{\res}(p)\cong\gall(T)$. The abelianization of $\gall^1(p)$ (i.e., the group $\gall^1(p)/(\gall^1(p))'$) is isomorphic to the first homology group $H_1(p)$.
\end{Fact}

%\begin{Fact}
%$H_1(p)$ is isomorphic to  $G/K$ where $G:=\aut(p)$ and $K$ is the normal  subgroup of $G$ consisting of all automorphisms fixing all
%orbits of elements of $p(\CM)$ under the action of $G'$ (so $G/K$ is independent from the choice of a monster model). 
%\end{Fact}

The last statement in the above fact explains the connection between the relativized Lascar group and the model theoretic homology group $H_1$ of the type. 
Let us recall now the key definitions in the homology theory in model theory. We fix a ternary automorphism-invariant relation $\indo^*$ between
small sets of $\CM$ satisfying
\begin{itemize}
\item
finite  character: for any sets $A,B,C$, we have $A\indo^*_CB$ iff $a\indo^*_Cb$ for any finite tuples $a\in A$ and $b\in B$;

\item normality: for any sets $A$, $B$ and $C$, if $A\indo^* _C B$, then $A\indo^*_C \acl(BC)$;

\item
symmetry: for any sets $A,B,C$, we have $A\indo^*_CB$ iff $B\indo^*_CA$;
\item
transitivity: $A\indo^*_BD$ iff $A\indo^*_BC$ and $A\indo^*_CD$,  for any sets $A$ and
$B\subseteq	 C\subseteq D$;
\item
 extension: for any sets $A$ and $B\subseteq C$, there is $A'\equiv_BA$ such that
 $A'\indo^*_BC$.
%\item normality: for any sets $A$, $B$ and $C$, if $A\indo^* _C B$, then $A\indo^*_C \acl(BC)$.
\end{itemize}

\noindent Throughout this paper we call the above axioms {\bf the basic 5 axioms}.
We say that $A$ is $*$-independent from $B$ over $C$ if $A\indo^*_C B$. Notice that there is at least one such relation for any theory, namely, 
the {\em trivial independence relation} given by: For any sets $A,B,C$, put $A\indo^*_BC$. Of course there is a non-trivial such relation when $T$ is simple or rosy, given by forking or thorn-forking, respectively.

%Now, we also fix a strong type $p$ of possibly infinite arity over $B=\acl(B)$.
% We shall define the first homology group of $p$ with respect to $\indo^*$, analogously to that in the references.
% Hence we begin by  recalling  some notations from the references.
%recall the defintion of the first homology group of $p$.

\begin{notation}
Let $s$ be an arbitrary finite set of natural numbers. Given any subset
$X\subseteq \CP(s)$, we may view $X$ as a category where for any
$u, v \in X$, $\mor(u,v)$ consists of a single morphism 
$\iota_{u,v}$ if $u\subseteq v$, and $\mor(u,v)=\emptyset$ otherwise. 
If $f\colon X \raw \CC_0$ is any functor into some category $\CC_0$, then for any $u, v\in X$ with $u \subseteq v$, we let $f^u_v$ denote the morphism
$f(\iota_{u,v})\in \mor_{\CC_0}(f(u),f(v))$.
We shall call $X\subseteq \CP(s)$ a \emph{primitive category} if $X$ is non-empty and  \emph{downward closed}; i.e.,  for any $u, v\in \CP(s)$,  if $u\subseteq v$ and $v\in X$ then $u\in X$. (Note that all primitive categories have  the empty set $\emptyset\subseteq \omega$ as an object.)

We use now $\CC$ to denote the category  whose objects are the small subsets of $\CM$, and whose morphisms are elementary maps.
For a functor $f:X\to \CC$ and objects $u\subseteq v$ of $X$, $f^u_v(u)$ denotes the set $f^u_v(f(u))(\subseteq f(v))$.
\end{notation}

\begin{definition}\label{p-*-functors}
By a \emph{$*$-independent functor in $p$}, we mean a functor $f$ from some primitive category $X$ into $\CC$ 
satisfying the following:

\be
\item If $\{ i \}\subseteq \omega$ is an object in $X$, then $f(\{ i \})$ is of the form $\acl(Cb)$ where  
$b\models p$, $C=\acl(C)=f^{\emptyset}_{\{ i \}}(\emptyset)$,  and $b\indo^* C$.

\smallskip

\item Whenever $u(\neq \emptyset)\subseteq \omega$ is an object in $X$, we have
\[ f( u) = \acl \left( \bigcup_{i\in u} f^{\{ i \}}_u(\{ i \}) \right)
\]
and $\{f^{\{i\}}_u(\{i\})|\ i\in u \}$ is $*$-independent over $f^\emptyset_u(\emptyset)$.
\ee
%If $f$ is a closed independent functor such that $f(\emptyset) = B$, we shall say that $f$ is \emph{over $B$}.
 We let $\CA^*_p$  denote the family of all $*$-independent functors in $p$.

 A $*$-independent functor $f$  is called a {\em $*$-independent $n$-simplex} in $p$  if $f(\emptyset)=\emptyset$, our named algebraically closed set, and $\dom(f)=\CP(s)$ with $s\subseteq \omega$ and  $|s|=n+1$.  We call $s$ the {\em support} of $f$ and denote it by $\supp(f)$.
\end{definition}

In the rest we may call a  $*$-independent $n$-simplex in $p$ just  an {\em $n$-simplex} of $p$, as far as no  confusion arises.
%We are ready to define the first homology group $H^*_1(p)$ of $p$ depending on our choice of the independence relation $\indo^*$.

\begin{definition} Let $n\geq 0$.
 We define:
\begin{align*}
&S_n(\CA^*_p):= \{\, f\in \CA^*_p \mid \mbox{$f$ is an $n$-simplex  of $p$}\, \}\\
&C_n(\CA^*_p): =  \mbox{the free abelian group generated by $S_n(\CA^*_p)$.}
\end{align*}
An element of $C_n(\CA^*_p)$ is called an \emph{$n$-chain} of $p$.
The support of a chain
$c$, denoted by $\supp(c)$, is the union of the supports of all the simplices that appear in $c$ with a non-zero
coefficient.
Now for $n\geq 1$ and  each $i=0, \ldots,  n$, we  define a group homomorphism
\[ \Bd^i_n\colon C_n(\CA^*_p) \raw C_{n-1}(\CA^*_p)
\]
by  putting, for any $n$-simplex   $f\colon \CP(s) \raw \CC$ in $S_n(\CA^*_p)$ where  $s = \{ s_0 <\cdots < s_n \}\subseteq \omega$,
\[ \Bd_n^i(f):= f\restriction  \CP(s\setminus\{s_i\})
\]
and then extending linearly to all $n$-chains in $C_n(\CA^*_p)$. Then we define  the \emph{boundary map}
\[ \Bd_n\colon C_n(\CA^*_p)\raw C_{n-1}(\CA^*_p)
\]
by
\[ \Bd_n(c):=\sum_{0\leq i\leq n}(-1)^i \Bd_n^i(c).
\]
We shall often refer to $\Bd_n(c)$ as the \emph{boundary of $c$}. Next, we define:
\begin{align*}
&Z_n(\CA^*_p):= \ker \, \Bd_n\\
& B_n(\CA^*_p):= \Im \, \Bd_{n+1}.
\end{align*}

\noindent The elements of $Z_n(\CA^*_p)$ and $B_n(\CA^*_p)$ are called \emph{$n$-cycles} and \emph{$n$-boundaries} in $p$, respectively.
It is straightforward to check that
$ \Bd_{n}\circ \Bd_{n+1} =0.$
Hence we can now define the group
\[ H^*_n(p):= Z_n(\CA^*_p)/B_n(\CA^*_p)
\]
called the \emph{$n$}th \emph{$*$-homology group} of $p$.
\end{definition}

\begin{notation}
\be
\item
For $c\in Z_n(\CA^*_p)$, $[c]$ denotes the homology class of $c$ in $H^*_n(p)$.
\item When $n$ is clear from the context, we shall often omit it in $\Bd^i_n$ and in $\Bd_n$, 
writing simply  $\Bd^i$ and $\Bd$.
\ee
\end{notation}

\begin{definition}
A $1$-chain  $c\in C_1(\CA^*_p)$ is called a \emph{$1$-$*$-shell} (or just a $1$-shell) in $p$ if it is of the form
\[c =f_0-f_1+f_2
\]
where $f_i$'s are $1$-simplices of $p$ satisfying
\[ \Bd^i f_j = \Bd^{j-1} f_i \quad \mbox{whenever $0\leq i < j \leq 2$.}
\]
Hence, for $\supp(c)=\{n_0<n_1<n_2\}$ and $ k\in \{0,1, 2\}$, it follows that
$$\supp(f_k)=\supp(c)\smallsetminus \{n_k\}.$$
\end{definition}

\noindent Notice that the boundary of any $2$-simplex is a $1$-shell. Recall that 
 a notion of  an {\em amenable} collection of functors into
a category is introduced in \cite{GKK}.
 Due to the 5 axioms of $\indo^*$, it easily follows that
$\CA^*_p$ forms such a collection of functors into $\CC$. Hence the following corresponding fact holds. 

\begin{fact}\label{H1 in amenable family} (\cite{GKK} or \cite{DKL})
\[ H^*_1(p) = \{ [c] \mid c \mbox{ is a } \mbox{$1$-$*$-shell with}\ \supp(c) = \{ 0, 1, 2 \} \, \}.
\]
\end{fact}

We now recall basic notions and results appeared in \cite{DKL}. 
{\bf For the rest of this section we assume that $p$ is the strong type of
an algebraically closed set.}
%Next, we introduce the notion of representation of $1$-shells.

\begin{definition}\label{representation}
\be\item Let $f:\CP(s)\to \CC$ be an $n$-simplex of $p$.
For $u\subseteq s$ with $u = \{ i_0 <\ldots <i_k \}$, we shall write $f(u)=[a_0 \ldots a_k]_u$, where each $a_j\models p$ is 
an algebraically closed tuple as assumed above, if $f(u)=\acl(a_0\ldots a_k)$, and $\acl(a_j)=f^{ \{ i_j \} } _u (\{i_j\})$. So, $\{a_0,\ldots,a_k\}$ is $*$-independent.
Of course, if we write $f(u)\equiv [b_0 \ldots b_k]_u$, then it means that there is an automorphism sending $a_0\ldots a_k$ to $b_0\ldots b_k$. 
\item
Let $s=f_{12}-f_{02}+f_{01}$ be a $1$-$*$-shell in $p$ such that $\supp(f_{ij})=\{n_i,n_j\}$ with $n_i<n_j$ for $0\le i<j\le 2$. Clearly there is a quadruple $(a_0,a_1,a_2,a_3)$ of realizations of $p$ such that $f_{01}(\{n_0,n_1\})\equiv[a_0 a_1]_{\{n_0,n_1\}}$, $f_{12}(\{n_1,n_2\})\equiv[a_1 a_2]_{\{n_1,n_2\}}$, and $f_{02}(\{n_0,n_2\})\equiv[a_3 a_2]_{\{n_0,n_2\}}$. We call this quadruple a {\em  representation of $s$}. For  any such representation of $s$,  call $a_0$  an {\em  initial point}, $a_3$ a {\em terminal point},  and $(a_0,a_3)$ an {\em  endpoint pair} of the representation.
\ee
\end{definition}

\noindent We summarize some properties of endpoint pairs of $1$-shells. We define an equivalence relation $\sim^*$ on the set of pairs of realizations $p$ as follows: For $a,a',b,b'\models p$, $(a,b)\sim^*(a',b')$ if two pairs $(a,b)$ and $(a',b')$ are endpoint pairs of $1$-shells $s$ and $s'$ respectively such that $[s]=[s']\in H^*_1(p)$. 
%Then $\sim^*$ forms an equivalence relation on $p(\CM)\times p(\CM)$.
We write $\CE^*=p(\CM)\times p(\CM)/\sim^*$. We denote the class of $(a,b)\in p(\CM)\times p(\CM)$  by $[a,b]$. Now, define a binary operation $+_{\CE^*}$ on $\CE^*$ as follows: For $[a,b],[b',c']\in \CE^*$, $[a,b]+_{\CE^*} [b',c']=[a,c]$, where $bc\equiv b'c'$.
\begin{Fact}\label{summary_H1_representation}
The operation $+_{\CE^*}$ is well-defined, and the
pair $(\CE^*,+_{\CE^*})$ forms an abelian group which is isomorphic to $H^*_1(p)$.
More specifically, for $a,b,c\models p$ and $\sigma\in \aut(\CM)$, we have:
\begin{itemize}
	\item $[a,b]+[b,c]=[a,c]$;
	\item $[a,a]$ is the identity element;
	\item $-[a,b]=[b,a]$;
	\item $\sigma([a,b]):=[\sigma(a),\sigma(b)]=[a,b]$; and
	\item $f:\CE^*\to H^*_1(p)$ sending $[a,b]\mapsto [s]$, where $(a,b)$ is an endpoint pair of $s$, is a group isomorphism.
\end{itemize}
We identify $\CE^*$ and $H^*_1(p)$.
\end{Fact}
\begin{fact}
$H_1^*(p)$ is isomorphic to  $G/N$ where $G:=\aut(p)$ and $N$ is the normal  subgroup of $G$ consisting of all automorphisms fixing setwise all
orbits of elements of $p(\CM)$ under the action of $G'$ (so $G/N$ is independent from the choice of a monster model). Hence $H_1^*(p)$ is independent from the choice of $\indo^*$ and we write $H_1(p)$ for $H_1^*(p)$.

\end{fact}

\begin{Fact}\label{H1properties}  Let $p$ be the fixed  strong type of an algebraically closed set.
\be 
\item Let $\indo^*$ be an independence relation satisfying the 5 basic axioms. Let $a,b \models p$.
Then the following are equivalent.
\be
\item $[a,b]=0\in H_1(p)$;
\item There is a balanced-chain-walk from $a$ to $b$, i.e., there are some $n\geq 0$ and a  finite sequence $(d_i)_{0\le i\le 2n+2}$ of realizations of $p$ satisfying the following conditions:
\be
	\item $d_0=a$, and $d_{2n+2}=b$;
	\item $\{d_{j},d_{j+1}\}$ is $*$-independent for each $j\le 2n+1$; and
	\item there is a bijection $$\sigma:\{0,1,\ldots,n\}\to\{0,1,\ldots,n\}$$ such that
	$d_{2i} d_{2i+1}\equiv d_{2\sigma(i)+2}d_{2\sigma(i)+1}$ for $ i \le n$;
\ee

\item  There are some $n\geq 0$ and   finite sequences $(d_i:0\le i\le n)$, $(d^j_i: i< n,\ 1\le j\le 3)$ of realizations of $p$
such that $d_0=a$, $d_n=b$, and for each $i<n$,  $d_id^1_i\equiv d^3_id^2_i$, $d^1_id^2_i\equiv d_{i+1}d^3_i$. 
\ee
In particular, if $a\equiv^L b$ then $[a,b]=0\in H_1(p)$.

\item The following are equivalent. 
\be\item $\gall^1(p)$ is abelian;
\item For all $a,b \models p$,  $[a,b]=0$ in $H_1(p)$  iff $a \equiv^L b$. 
\ee
\item $p$ is a Lascar type (i.e. $a\equiv^Lb$ for for any $a,b\models p$) iff $H_1(p)$ is trivial and $\gall^1(p)$ is abelian. 
\ee
\end{Fact}

\begin{Remark}\label{modelinterdef}
\be
\item Assume that $p(x)$ is the type of a small model. Then for any $M, N\models p$, the equivalence classes of (equality of) Lascar types of $M$ and $N$ are interdefinable:  
Let $M\equiv^L M'$ and $MN\equiv M'N'$. It suffices to show that $N\equiv^L N'$.
Now there is $f\in \autf(\CM)$ such that $f(M')=M$. Let $N'':=f(N')$ so that $N''\equiv^LN'$ and $MN\equiv MN''$.
Hence $N\equiv_{M}N''$ and $N\equiv^L N''\equiv^L N'$ as wanted. 
When $p$ is  a type  of any tuple we will see that the same holds if $\gall^1(p)$ is abelian (Remark \ref{abel.interdef}).

\item
Notice that for tuples $a,b\models p$, there is a commutator $f$ in $\aut(p)$ such that  $f(a)=b$ if and only if there are $d^1,d^2,d^3\models p$ such that 
$ad^1\equiv d^3d^2$ and $bd_3\equiv d_1d_2$ \ (**). Fact \ref{H1properties}(1)(c) is an iterative application of this to that $[a,b]=0\in H_1(p)$ if and only if $h(a)=b$ for
some $h$  in the commutator subgroup of $\aut(p)$. 

Now we recall the following fact of compact group theory by M. Got\^{o} from \cite[Theorem 9.2]{HM}:
 Assume  $(F,\cdot )$ is a compact connected topological group. Then
$F'$, the commutator subgroup of $F$, is  simply the set of commutators in $F$, i.e. $$F'=\{ f\cdot g \cdot f^{-1}\cdot g^{-1}\mid f,g \in F\}.$$  
It follows that $F'$ is a closed subgroup of $F$, and  both $F'$ and $F/F'$ are  compact connected groups as well. 

Due to the theorem   we newly observe here that if $T$ is $G$-compact then in Fact \ref{H1properties}(1)(c), we can choose $n=1$: There are $d^1,d^2,d^3\models p$
 such that   above (**) holds.
\ee
\end{Remark}

\section{Amalgamation properties of strong types in simple theories}

Note that if $T$ is simple then since $T$ is $G$-compact, each $\gall^{\lambda}(p)$ is a compact
(i.e., quasi-compact and Hausdorff) connected group, so it is divisible (see \cite[Theorem 9.35]{HM}).
{\bf In this section we assume $T$ is simple} (except  Remark \ref{abel.interdef} and  
Example \ref{failure_reversibale_amalgam_rosy}), and the independence  is  nonforking independence. 
It is still an open question whether the strong type $p$ is necessarily a Lascar type. If so, then the
following theorem follows easily by the $3$-amalgamation of Lascar types in simple theories.  
But regardless of the answer to the question, $p$ has the following amalgamation property.
We {\em do not assume} here that a realization of  $p$ is algebraically closed. 

\begin{Theorem}[divisible amalgamation] \label{divamal} Let $p$ be a strong type in a simple theory. Let $a,b\models p$ and $a\indo b$. Then for each $n\geq 1$, there are independent $a=a_0,a_1,\ldots,a_n=b$ such that
$a_0a_1\equiv a_ia_{i+1}$ for every $i<n$. 
\end{Theorem}
\begin{proof}
Clearly we can assume $n>1$.  
Note that there is $f \in \aut(p)$ such that $b=f(a)$. Then since $G:=\gall^1(p)$ is divisible, there is $h \in \aut(p)$ such that
$[h]^n=[f]$ (in $G$). Put $c=h(a)$. 
Now there is $c_1\equiv^L c$ such that $c_1 \indo ab$. Then there is $h' \in \autf^1(p)$
such that $h'(c)=c_1$. Let $g=h'\circ h$. Then $[g]=[h]$ so $[g]^n=[f]$ too, and $g(a)=c_1$. 

\medskip

\noindent{\em Claim.} We can find additional elements $c_2,\ldots,c_n$ such that $\{a=c_0,c_1,\ldots,c_n,b\}$ is independent, and for each $1\leq m\leq n$,  $c_0c_1\equiv c_{m-1}c_{m}$  and there is $h_m\in\aut(p)$ such that 
$[h_m]=[g]^m$ in $G$, and $h_m(a)=c_m$:   
For an induction hypothesis, assume 
for $1\leq m<n$ we have found
$a=c_0,c_1,\ldots, c_m$ such that $\{c_0,c_1,\ldots,c_m,b\}$ is independent and $c_0c_1\equiv c_{i-1}c_{i}$ for all $1\leq i\leq m$ and there is $h_m\in\aut(p)$ such that 
$[h_m]=[g]^m$ in $G$, and $h_m(c_0)=c_m$.  

Notice now that then $c'_m:=g^m(a)\equiv^L h_m(a)=c_m$.
Put $c'_{m+1}:=g(c'_m)=g^{m+1}(a)$. 
Then there is $c_{m+1}$ such that $c_mc_{m+1}\equiv^L c'_mc'_{m+1}$, and $\{c_0,c_1,\ldots,c_{m+1},b\}$ is independent.
Since $c_{m+1}\equiv^L c'_{m+1}$,  there is $h'' \in \autf^1(p)$, such that
$h''(c'_{m+1})=c_{m+1}$ and so for $h_{m+1}:=h''\circ g^{m+1}$, we have $h_{m+1}(a)=c_{m+1}$ and  $[g]^{m+1}= [h_{m+1}]$ in $G$.
Moreover the equality of types $c_mc_{m+1}\equiv c'_mc'_{m+1} \equiv c_0c_1$ is witnessed by $g^m$. Hence the claim is proved.

\medskip

Notice that  $ c_n=h_n(a)$ with $[h_n]=[g]^n=[f]$. Hence $c_n\equiv^L b=f(a)$.
Then, by $3$-amalgamation, we find $b' \models \tp(b/a) \cup \tp(c_n/c_1\dots c_{n-1})$
with $b'\indo c_0\ldots c_{n-1}$. Then the automorphic images of
$c_1\dots c_{n-1}$ (rename them as $a_1\ldots a_{n-1}$) under a map sending $b'$ to $b$ over $a$ satisfy
the conditions of the theorem.
\end{proof}

Now we assume that $p$ is the type of an algebraically closed tuple. For $a,a'\models p$, we have a better description of $[a,a']=0\in H_1(p).$

\begin{proposition}\label{indep.commu.} For $a,a'\models p$, the following are equivalent.
\be\item
$[a,a']=0\in H_1(p)$, equivalently there is $h$ in the commutator subgroup of $\aut(p)$ such that $h(a)=a'$. 
\item
There are $b,c,d\models p$ such that each of $\{a,b,c,d\}$, $\{a', b,c,d\}$ is independent, and 
$ab\equiv  cd$,  $bd\equiv a'c$.
\ee
\end{proposition}
\begin{proof}
(1)$\Rightarrow$(2) Since $T$ is $G$-compact, by  Remark \ref{modelinterdef}(2) there are 
$b',c',d'\models p$ such that 
$ab'\equiv c'd'$ and   $b'd'\equiv a'c'$.  Now there is $b\equiv^L b'$ such that
$b\indo aa'b'c'd'$. Hence, by $3$-amalgamation, there 
is $d_0\equiv^L d'$ such that $d_0\models \tp(d'/b) \cup \tp(d''/b')$ and $d_0\indo bb'$,
 where $d''b'\equiv d'b$ and $d''\equiv^L d'\equiv^L d_0$. By extension there is no
 harm in assuming that $d_0\indo bb'aa'c'd'$.
 
 Now since $a'c'\equiv b'd'$, there are $a''c$ such that $b'd'bd_0\equiv a'c'a''c$. Hence, 
 $c\equiv^L c'$, $a''\equiv^L a'$, and $c\indo a'c'$. Again by extension, we can further assume that $c\indo aa'c'bd_0$.
 Moreover, since $ab'\equiv c'd'$, there is $d_1\equiv^L d'$ such that $ab'b\equiv c'd'd_1$.
 Then since $c\indo c'$ with $c\equiv^Lc$, and $d_1\indo c'$, again by $3$-amalgamation we can assume that $c'd_1\equiv cd_1$ and $d_1\indo cc'$.
 
 Now the situation is that $aa'b\indo c$, $d_1\indo c$, and $d_0\indo aa'b$. Moreover, $d_0\equiv^L d'\equiv^L d_1$.
 Hence, by $3$-amalgamation,
 we have $d\models \tp(d_0/aa'b)\cup \tp(d_1/c)$ and $d\indo aa'bc$. 
 Therefore, each of  $\{a,b,c,d\}$, $\{a', b,c,d\}$ is independent.
 Moreover, due to above combinations  $$ab\equiv  c'd_1\equiv cd_1\equiv cd,$$
 and   
 $$bd\equiv bd_0\equiv bd'\equiv b'd''\equiv b'd_0\equiv a'c$$
  as wanted.

(2)$\Rightarrow$(1) Clear by Remark \ref{modelinterdef}(2).
\end{proof}

%{\bf For the rest of this section} (except   Remark \ref{abel.interdef} and  Example \ref{failure_reversibale_amalgam_rosy}) 
{\bf From now until Theorem \ref{lascar=rev.}, we assume that 
$p(x)$ is the (strong) type of a small model, and all tuples    $a,b,c,\dots$ realize  $p$}, so all are universes of  models.
 Hence, by Fact \ref{rel.lascargp}, $\autf^1(p)=\autf_{\fix}(p)$ and $(\gall(T)\cong )\gall^{\fix}(p)=\gall^{1}(p)$, which we simply 
write $\autf(p)$ and $\gall(p)$, respectively.  Moreover,  $H_1(p)\cong \gall(p)/(\gall(p))'$ and hence it is a compact  connected (so divisible) abelian group.

  %the essential idea is similar.  mm

\begin{Definition}
Let $r(x,y), s(x,y)$ be types completing $p(x)\wedge p(y)\wedge x\indo y$. 
\be\item
We say $p$ has {\em abelian} (or {\em commutative}) {\em amalgamation of $r$ and $s$}, if 
there are independent $a,b,c,d\models p$ such that $ab,cd\models r$ and 
$ac, bd\models s$. 

We say $p$ has {\em abelian amalgamation} if so has it for any such completions $r$ and $s$.

\item

We say $p$ has {\em reversible amalgamation of $r$ and $s$} if 
there are independent $a,b,c,d\models p$ such that $ab,dc\models r$ and 
$ac, db\models s$. 

We say $p$ has {\em reversible amalgamation} if so has it for any such completions $r$ and $s$.

\ee
\end{Definition}

\begin{Lemma}\label{strongabelamal} The following are equivalent. 
\be
\item
The type $p$ has abelian amalgamation.

\item
Let $r(x,y), s(x,y)$ be any types completing $p(x)\wedge p(y)\wedge x\indo y$, and let
$a,b,c\models p$ be independent such that $ac\models s$ and $ab\models r$. Then there is 
$d\models p $ independent from $abc$ such that $cd\models r$ and $bd \models s$. 

\ee

\end{Lemma}
\begin{proof}
(1)$\Rightarrow$(2) By (1), there are $b_0, d_0\models p$ such that 
$\{a,c,b_0,d_0\}$ is independent, $ab_0,cd_0\models r$ and $ac, b_0d_0\models s$.
Hence, there is $d_1$ such that $abd_1\equiv ab_0d_0$.
Now by $3$-amalgamation over the model $a$, there is
$$d\models \tp(d_0/a;c) \cup \tp(d_1/a;b),$$
such that $\{a,b,c,d\}$ is independent. 
Moreover, $cd\equiv cd_0\models r$ and $bd\equiv bd_1\equiv b_0d_0 \models s$,  as desired.

(2)$\Rightarrow$(1) Clear.
\end{proof}

By the same proof we obtain the following too.

\begin{Lemma}\label{strongrevamal} The following are equivalent. 
\be
\item
The type $p$ has reversible amalgamation.

\item
Let $r(x,y), s(x,y)$ be any types completing $p(x)\wedge p(y)\wedge x\indo y$, and let
$a,b,c\models p$ be independent such that $ac\models s$ and $ab\models r$. Then there is 
$d\models p $ independent from $abc$ such that $dc\models r$ and $db \models s$. 

\ee

\end{Lemma}

\begin{theorem}\label{abeliangal(p)} The following are equivalent. 
\be
\item
$ \gall(p)(\cong\gall(T))$ is abelian.

\item

$p$ has abelian amalgamation.
\ee
\end{theorem}
\begin{proof}
(1)$\Rightarrow$(2) Assume (1). Thus for $G:=\aut(p)$, we have 
 $G'\leq \autf(p)$ \ \ (*). Now let $r(x,y)$ and $s(x,y)$ be some complete types containing 
 $p(x)\wedge p(y)\wedge x\indo y$. There are independent $a,a', b,d,c\models p$ such that
 $ab,cd\models r$ and $bd, a'c\models s$. Hence there is a commutator $f$ in $G'$ such that
 $f(a)=a'$. Hence by (*), $a\equiv^L a'$. 
 %Then since they are independent, they begin a Morley sequence.
 %In particular there is $a''\models p$ such that $\{a,a',a'', b,c,d \}$ is independent and $aa''\equiv a'a''$.  
Then by $3$-amalgamation of $\Lstp(a)$, there is 
$$a_0\models \tp(a/bd)\cup  \tp(a'/c)$$
such that $\{a_0,b,c,d\}$ independent; $a_0b\equiv ab\models r$; and $a_0c\equiv a'c\models s$. 
Therefore $p$ has commutative amalgamation.

(2)$\Rightarrow$(1) 
To show (1), due to Fact \ref{H1properties}(2) it is enough to prove that  if  
$a,a'\models p$ and $[a,a']=0\in H_1(p)$ \ ($\dag$), then $a\equiv^L a'$. Now assume (2) and ($\dag$).
Thus, by Proposition \ref{indep.commu.}, there are $b,c,d\models p$ such that each of 
the sets $\{a,b,c,d\}$ and $\{a', b,c,d\}$ is independent, 
$ab\equiv  cd$ and  $bd\equiv a'c$. Now by Lemma \ref{strongabelamal} and extension there is $c'\models p$ such that $c'\indo aa'cd$ and 
$ac'\equiv bd\equiv a'c$ and $c'd\equiv ab\equiv cd$.  Then, since $d$ is a model, we have $c
\equiv^L c'$, and there is $a''$ such that $a'c\equiv_da''c'$, so $a\equiv^L a''$ and $a''\equiv_{c'}a$. Since
$c'$ is a model as well, we conclude that $a'\equiv^L a''\equiv^L a$. 
\end{proof}

A  proof similar to that of the above Theorem \ref{abeliangal(p)} (2)$\Rightarrow$(1) applied to Fact \ref{H1properties}(1)(b) gives the following  as well. 

\begin{Proposition}\label{revimpliesabel}
If $p$ has reversible amalgamation, then $\gall(p)$ is abelian (equivalently, $p$ has abelian amalgamation).
\end{Proposition}
\begin{proof} Assume that $p$ has reversible amalgamation, and   $[a,b]=0\in H_1(p)$  for  $a,b\models p$\ ($\dag$). As before it is enough to prove 
that $a\equiv^L b$. 
It follows from the extension axiom for Lascar types together with Fact \ref{H1properties}(1)(b), 
there are $b'\equiv^L b$ and a finite {\em independent} sequence $(d_i)_{0\le i\le 2n+2}$ of  realizations of $p$ satisfying the following conditions:
\be
	\item[(i)] $d_0=a$,  $d_{2n+2}=b'$; and 
	%\item[(ii)] $\{d_{j},d_{j+1}\}$ is independent for each $j\le 2n+1$; and
	\item[(ii)] there is a bijection $$\sigma:\{0,1,\ldots,n\}\to\{0,1,\ldots,n\}$$ such that
	$d_{2i} d_{2i+1}\equiv d_{2\sigma(i)+2}d_{2\sigma(i)+1}$ for $ i \le n$.
\ee
In other words, there is an independent balanced chain-walk from $a$ to $b'$, and  it suffices to prove $a\equiv^L b'$.
Notice that reversible amalgamation implies  Lemma \ref{strongrevamal}(2) which exactly means that   in above chain-walk two adjacent edges can be swapped with sign 
reversed (i.e.,  for say $d_jd_{j+1}$ and $d_{j+1}d_{j+2}$ there is $d'_{j+1}(\indo d_jd_{j+2})$ such that $d_jd'_{j+1}\equiv d_{j+2}d_{j+1}$ and $d'_{j+1}d_{j+2}\equiv d_{j+1}d_{j}$). 
By iterating this process one can transfer the chain-walk to another balanced chain-walk $(d'_i)_{0\le i\le 2n+2}$ from $a$ to $b'$ such that the walk has 
a Lascar pattern, i.e., the bijection
$\sigma$ for the new walk is the identity map. Hence it follows $d'_{2i}\equiv_{d'_{2i+1}} d'_{2i+2}$ for $i\leq n$, and since each $d'_{2i+1}$ is a model,
we have $a=d'_0\equiv^Lb'$, as wanted.
\end{proof}

A stronger consequence is obtained.

\begin{Theorem}\label{lascar=rev.}
The strong type $p$ is a Lascar type iff $p$ has reversible amalgamation. 
\end{Theorem}  
\begin{proof}
($\Rightarrow$) It follows from $3$-amalgamation of Lstp. 

($\Leftarrow$) Assume $p$ has reversible amalgamation. Due to Fact \ref{H1properties}(3) and Proposition \ref{revimpliesabel}, it suffices to show $H_1(p)$ is trivial.
Let an arbitrary  $[s]\in H_1(p)$ be given, and let $[a,b']=[s]$ for $a,b'\models p$.  Then, for any $b\equiv^L b'$ with $b\indo ab'$, we have $[a,b]=[a,b']+[b',b]=[s]+0=[s]$.
  Similarly, for  $c\equiv^L b$ with  $c\indo ba$, we have  $[b,c]=0$ in $H_1(p)$.
Now reversible amalgamation (or Proposition \ref{strongrevamal}) says that $[s]=[a,b]+[b,c]=[c,b]+[b,a]=-[s]$. Hence $[s]+[s]=0$. Since $H_1(p)$ is compact and connected so divisible, any element in $H_1(p)$ is divisible by $2$. 
   Therefore, $H_1(p)=0$ by what we have just proved.
\end{proof}

\begin{Question} Can the same results hold if $p$  is the type of an
algebraically closed tuple (not necessarily a model) in a simple theory? 
The answer to this question is yes if  any two Lascar equivalence classes  in $p$  are interdefinable, since essentially this property (Remark \ref{modelinterdef}(1)) implied 
the results in this section when $p$ is the type of a model.  At least we can show the following remark.
\end{Question}

\begin{remark} \label{abel.interdef}
Let $T$ be any theory, and  let a realization of  $p$ be any tuple. If $\gall^1(p)$ is abelian then any two Lascar equivalence classes in $p$ are interdefinable: Let $a,b\models p$, and $f\in \aut(p)$.
Assume $f(a)\equiv^L a$. We want to show the same holds for $b$. Now there is $g\in \aut(p)$ such that $g(a)=b$. Since $\gall^1(p)$ is abelian, we have 
$f(b)=f(g(a))\equiv^L g(f(a))\equiv^L g(a)=b$. 

Notice that 
  $\gall^1(p)$ is the group of automorphic permutations of the Lascar classes in $p$. Hence,
  if  $\gall^1(p)$ is abelian, then $f/\autf^1(p)\in \gall^1(p)$ is determined by
the pair of Lascar classes of $c$ and  $f(c)$ for some (any) $c\models p$.  
\end{remark}

\begin{example}\label{failure_reversibale_amalgam_rosy}
Let $(\CM,<)$ be a monster model of $\Th(\BQ,<)$. Then   thorn-independence 
$\indo^*$ in $\CM$  coincides with  $\acl$-independence. We can also consider a notion of
reversible amalgamation using thorn independence instead of nonforking independence. 
Note that the Lascar  group of $\Th(\CM)$ is trivial because of $o$-minimality, but the 
reversible amalgamation for $\indo^*$ fails.

Let $p$ be the unique $1$-type over $\emptyset(=\acl(\emptyset))$. Consider two types $r(x,y)=\{x<y\}$ and $s(x,y)=\{y<x\}$ completing $p(x)\wedge p(y)\wedge x\indo^* y$. Then $p$ has no reversible amalgamation of $r$ and $s$. Suppose $a,b,c,d\models p$ such that $ab,dc\models r$ and $ac,db\models s$. Then $a<b<d<c<a$, and there are no such elements. Thus, a type of a model  does not have reversible amalgamation either. 
\end{example}

\section{Relativized  Lascar Galois groups and algebraicity} 

Recall that throughout we assume that $\acl(\emptyset)=\dcl(\emptyset)$.
In Remark 1.8 from \cite{DKL} it was noticed that if  $\tp(a)$ is a Lascar type, and a finite tuple $c$ is algebraic over $a$ then 
$\tp(ac)$ is also a Lascar type, and  hence if additionally $T$ is $G$-compact then
 $\tp(\acl(a))$ is also a Lascar type. In other words, if $\gall^1(\tp(a))$ is trivial,
then $\gall^1(\tp(\acl(a))$ is trivial when $T$ is $G$-compact. It seems natural to ask whether, more generally,
$\gall^1(\tp(a))$ and
  $\gall^1(\tp(\acl(a))$ must be always isomorphic in a $G$-compact $T$. In general, the answer turns
  out to be negative (see Section 3), 
  but we obtain some positive results if, instead of looking at $\acl(a)$, we add only a finite part of it to $a$.

For the whole Section 2, we fix the following notation.
\begin{notation} Assume 
  that $p=\tp(a)$ and $\bar{p}=\tp(ac)$, where 
$c\subseteq \acl(a)$ is finite. Consider the natural projection
$\pi_{\lambda}:\gall^{\lambda}(\bar{p})\to \gall^{\lambda}(p)$, and  put $ \pi:=\pi_1$. We denote the  kernel of $\pi$ by $K$.
Denote by $E$ the relation of being Lascar-equivalent (formally, $E$ depends on 
the length of tuples on which we consider it).
 Let  $ac_1,\dots,ac_n$ be a tuple of representatives of all $E$-classes in $\p$ in which the first coordinate
 is  equal to $a$ (so $n\leq $ the number of realizations of $\tp (c/a)$), and 
 for any $a'\models p$, let  $a'c^{a'}_1,\dots,a'c^{a'}_n$ be its conjugate by an 
 automorphism sending $a$ to $a'$. 
\end{notation}

\begin{example}\label{g1neg2}  We give an example of a structure and a type $p$ such that $\gall^1(p)$ and $\gall^2(p)$ are distinct. Let $P$ be a Euclidean plane, say $\mathbb{R}\times \mathbb{R}$. For $n>0$,  define $R_n(xy, zw)$ on $P$ such that $R_n(ab, cd)$ iff $a\ne b$, $c\ne d$, and either lines $L(ab)$ containing $a,b$ and $L(cd)$ are parallel, or the smaller angle between $L(ab)$ and $L(cd)$ is $\leq \frac{\pi}{2n}$.  Consider a model $M=(P; R_n(xy,zw))_{0<n}$. Then $F(xy,zw):=\bigwedge_{0<n}R_n(xy,zw)$ is an $\emptyset$-type-definable bounded equivalence relation in  $T=\mbox{Th}(M)$, and each class corresponds to a class of lines whose slopes are infinitesimally close. 
%By a similar arguments analyzing examples in \cite{DKL}, it follows that  $T$
%has weak elimination of imaginaries, so $\emptyset=\acl^{\eq}(\emptyset)$.  
Let $p$ be the unique complete $1$-type over $\emptyset$. Indeed $p$ is a Lascar type, so $\gall^1(p)$ is trivial. On the  other hand, a rotation of (a saturated) $P$ around a point in $P$ belongs to $\autf^1(p)$ but does not belong to $\autf^2(p)$.  One can show that $\gall^2(p)$ is a circle group.
\end{example}

Let us recall basic  definitions and facts on covering maps between topological groups (see for example \cite{HM}).

\begin{remark}\label{covering}
\bi
\item Let $X, Y$ be topological spaces, and let  $f:X\to Y$ be a continuous surjective map. 
 We call  $f$ a {\em (k-)covering map} if  for each $y\in Y$ there is an open set $V$ containing $y$ such that $f^{-1}(V)$ is a union of ($k$-many) disjoint open sets in $X$, and $f$ induces a homeomorphism between   each such open set and $V$. We call $f$ a {\em local homeomorphism} if for any $x\in X$, there is an open set $U$ containing $x$ such that $f\restriction U:U\to f(U)$ is a homeomorphism. Obviously a covering map is a local homeomorphism. Conversely, if both $X,Y$ are compact, then a local homeomorphism is a covering map. 
 We call $X$ a {\em covering space} of $Y$ if there is a covering map from $X$ to $Y$. 

\item Let $G$ be topological group, and let $H$ be a normal subgroup of $G$ such that $G/H$ with its quotient topology is also a topological group, and the projection map $\mbox{pr}: G\to G/H$ is open continuous homomorphism.  Recall that $\mbox{pr}$ is a covering map iff $H$ is a discrete subgroup.  If $G$ is compact,
 then $H$ is discrete iff $H$ is finite, and  hence iff   $\mbox{pr}$ is a covering map.
 
 Assume  $F$ is another topological group and $f:G\to F$ is a continuous surjective homomorphism. If $G$ is compact and $F$ is Hausdorff,  then $f$ induces an isomorphism between compact topological groups 
 $G/\ker(f)$ and  $F$.   
 
 \item   Assume $T$ is $G$-compact. By above,  $\pi_{\lambda}$ is a covering homomorphism iff the kernel of $\pi_{\lambda}$ is finite.
In particular, we shall show that $\pi =\pi_1$ is a covering homomorphism (Corollary \ref{picovering}). 

\ei
\end{remark}

If $T$ is $G$-compact then the Lascar equivalence $E$ is $\emptyset$-type-definable,
 and we can assume any formula $\varphi$ in $E$
is symmetric (i.e., $\varphi(\bar z,\bar w)\models \varphi(\bar w,\bar z)$).

\begin{lemma}\label{sep.fmla} Assume $T$ is $G$-compact.
 There is a formula $\alpha(xy,x'y')\in E(xy,x'y')$ such that if
  $\bar p(xy)\wedge \bar p(x'y') \wedge \alpha(xy,x'y')\wedge E(x,x')$ holds then $\models E(xy,x'y').$
\end{lemma}
\begin{proof}
 Notice that the type
$$E(xy,x'y')\wedge\exists zwz'w'\equiv xyx'y' 
(\bigvee_{ 1\leq i\ne j\leq n} (E(zw,ac_i)\wedge E(z'w',ac_j)))$$ is inconsistent,
and choose $\alpha(xy,x'y')\in E(xy,x'y')$ so that the type
$$\alpha(xy,x'y')\wedge\exists zwz'w'\equiv xyx'y' 
(\bigvee_{1\leq i\ne j\leq n} (E(zw,ac_i)\wedge E(z'w',ac_j)))$$ is inconsistent.
Now, if $\bar p(xy)\wedge \bar p(x'y') \wedge\alpha(xy,x'y')\wedge E(x,x')$ holds, then 
we can find $c'a'c''$ such that $ac'a'c'' \equiv xyx'y'$, and then $a'\equiv^L a$
so $a'c'' \equiv^L ac_i$ for some $i$, and also $ac' \equiv^L ac_j$ for some $j$, but, by
the choice of $\alpha$, $i=j$ so $ac'\equiv^L a'c''$, so $xy \equiv^L x'y'$.
\end{proof}

\begin{theorem}\label{$G$-comp}
If $T$ is $G$-compact, then $K$ is finite.
\end{theorem}
\begin{proof} Assume $T$ is $G$-compact.
Since $E$ is transitive, there is $\phi(xy,x'y')\in E(xy,x'y')$ such that 
$$\phi(xy,zw)\wedge \phi(x'y',zw)\vdash \alpha(xy,x'y')\ \ (*),$$
where $\alpha\in E$ is the formula given by Lemma \ref{sep.fmla}.

Let $(a_{\ell})_{\ell\in I}$ be a small set of representatives of $E$-classes of realizations of $p$.
For any $a'c'\models \p$,  there are $\ell\in I$ and $1\leq j\leq n$ such that 
$\models E(a'c',a_{\ell}c_j^{a_{\ell}})$, so $\models  \phi(a'c',a_{\ell}c_j^{a_{\ell}})$.
Hence, by compactness, there are $a_0,\dots,a_k\models p$ such that
$$\p(xy) \vdash \bigvee\{  \phi(xy,a_ic^{a_i}_j)\mid  i\leq k;\ 1\leq j\leq n \}\ \ (**). $$

\begin{Claim}
Let $f/\autf^1(\p)\in K$. For each $a_ic_j^{a_i}$ chosen above, there is a unique $a_ic^{a_i}_{j'}$ with $ 1\leq j'\leq n$
such that $\phi(f(a_ic^{a_i}_j), a_ic^{a_i}_{j'})$ holds. Such a $j'$ does not depend on the choice of 
a representative of $f/\autf^1(\p)$.
\end{Claim}
\begin{proof}
Notice that  $f/\Autf^1(\p)\in K$ implies $E(f(a_i),a_i)$,
so $E(f(a_ic^{a_i}_j), a_ic^{a_i}_{j'})$ holds for some $j'$.
Now if $\phi(f(a_ic^{a_i}_j), a_ic^{a_i}_{j''})$ holds as well, then due to $(*)$ (with the symmetry of the formulas) and 
Lemma \ref{sep.fmla}, we must have $j'=j''$. The second statement of the claim follows similarly.
\end{proof}

\begin{claim}
Let $f/\autf^1(\p)$, $g/\autf^1(\p)\in K$. Assume that the permutations of tuples
$a_ic_j^{a_i}$ by $f$ and $g$ described in above claim
 are the same. Then $f/\autf^1(\p)=g/\autf^1(\p)$:
\end{claim} 
\begin{proof}
Let $a'c'\models \bar p$. By $(**)$, there is some $a_ic_j^{a_i}$ such that 
$\phi(a'c', a_ic_j^{a_i})$ holds. Hence,  $\phi(f(a'c'), f(a_ic_j^{a_i}))$ and $\phi(g(a'c'), g(a_ic_j^{a_i}))$ hold.
Moreover, by the previous claim with our assumption, there is $ j'\leq n$
such that $\phi(f(a_ic^{a_i}_j), a_ic^{a_i}_{j'})$ and $\phi(g(a_ic^{a_i}_j), a_ic^{a_i}_{j'})$ hold. 

Then, again due to 
$(*)$ and Lemma \ref{sep.fmla},  $f(a_ic^{a_i}_j)\equiv^L g(a_ic^{a_i}_j)$
since $f(a_i)\equiv^Lg(a_i)\equiv^L a_i$. Hence, there is $h\in \autf^1(\bar p)$ such that 
$f(a_ic^{a_i}_j)=hg(a_ic^{a_i}_j)$. Now by $(*)$ again,
$\alpha(f(a'c'), hg(a'c'))$ holds, and, since $f(a')\equiv^Lhg(a')\equiv^L a'$,
 we have  $f(a'c')\equiv^L hg(a'c')\equiv^L g(a'c')$.
%we have that $\alpha (f(a'c'), g(a'c'))$. Thus by Lemma \ref{sep.fmla},  $f(a'c')\equiv^L g(a'c')$ since $f(a')\equiv^Lg(a')\equiv^L a'$.
We conclude that $f/\autf^1(\p)=g/\autf^1(\p)$.
\end{proof}

As there are  only finitely many  permutations of  $a_ic_j^{a_i}$ ($i\leq k$, $ 1\leq  j\leq n$), $K$ is finite.
\end{proof}

By Remark \ref{covering}, we have the following.

\begin{Corollary} \label{picovering}
If $T$ is $G$-compact then $\pi:\gall^{1}(\bar{p})\to \gall^{1}(p)$ is a covering homomorphism.
\end{Corollary}

\begin{example}
Consider the following model 
$$M=((M_1,S_1,\{g^1_n\mid 0<n\}),(M_2,S_2,\{g^2_n\mid 0<n\}),\delta),$$
 a 2-sorted structure. Here 
  $M_1, M_2$ are disjoint  unit circles. For $i=1,2$,  $S_i$ is a ternary relation on $M_i$ such that $S_i(a,b,c)$ holds iff $a,b,c$ are 
 distinct and $b$ comes before $c$ going clockwise around $M_i$ from $a$; and  $g^i_n$ is the clockwise rotation of $M_i$ by $\frac{2\pi}{n}$-radians.
 The map $\delta:M_1\to M_2$ is a double covering, i.e. if we identify each $M_i$ as the unit circle in $xy$-plane centered at $0$, then 
 $\delta$ is given by $(\cos t, \sin t)\mapsto (\cos 2t,\sin 2t).$ 
By arguments   similar to those described  in \cite{DKL} for $M_i$,  it follows that 
$T=\operatorname{Th}(M)$  is $G$-compact, and, in $T$,  $\emptyset=\acl^{\eq}(\emptyset)$.
Let $\CM_1,\CM_2$ be saturated models of $M_1,M_2$ respectively. For any $a_i,a'_i\in \CM_i$, we have $a_i\equiv a'_i$; and
$a_i\equiv^La'_i$ iff they are infinitesimally close. Now, given $a\in M_2$, there are two antipodal $c_1,c_2\in M_1$
such that $\delta(c_i)=a$. Then $c_i\in \acl(a)$ and $ac_1\equiv ac_2$, but $ac_1\not\equiv^L ac_2$. For $p=\tp(a)$ and $\bar p=\tp(ac_1)$, 
$\pi:\gall^{1}(\bar p)\to \gall^{1}(p)$ is a 2-covering homomorphism
of circle groups. Notice that for any $n>2$, 
$$\delta(y)=x \wedge \delta(y')=x' \wedge (S_1(y,y',g_n^1(y)) \vee S_1(y',y,g_n^1(y'))) $$
serves as the formula $\alpha(xy,x'y')$ in Lemma \ref{sep.fmla}.
\end{example}

\begin{question}
If $T$ is $G$-compact,  then is the kernel of   the projection $\pi_{\lambda}:\gall^{\lambda}(\bar{p})\to \gall^{\lambda}(p)$ finite as well for $\lambda >1$?
Is $T$ being $G$-compact  essential in Theorem \ref{$G$-comp}? 
\end{question}

We have a partial answer of the second question above. Namely, we also get that $K$ is finite when we replace the assumption of $G$-compactness
by abelianness of $\gall^1(\bar{p})$. 
%in Theorem \ref{$G$-comp}.

\begin{proposition}\label{abelian}
 If $\gall^1(\bar{p})$ is abelian, then $|K|=n$.
\end{proposition}
\begin{proof}
If $f/\autf^1(\bar{p})\in K$, then $f(a)\equiv^L a$, so there is another representative  $f'$ fixing 
 $a$. Now, since $\gall^1(\bar{p})$ is abelian, by Remark \ref{abel.interdef}, the class of an automorphism $f'$ fixing $a$ depends
 only on the number of Lascar types of $f'(ac)=af'(c)$ for which, due to our choice, there are exactly $n$ possibilities.
\end{proof}

We investigate the epimorphism $\pi$ further when there are nicer conditions with $G$-compact $T$, and  
we will obtain  a better version of Theorem \ref{$G$-comp}. Recall some known facts first.

\begin{remark} \label{factsontop}
\be\item
Let $X$ be an $\emptyset$-type-definable set, and  let $F$ be a bounded  $\emptyset$-type-definable equivalence relation on $X$. 
Recall  the  {\em logic topology} on $X/F$: A subset of $X/F$ is closed if and only if its pre-image in $X$  is type-definable over some parameters.  It follows that $X/F$ is compact with the logic topology.  
\item 
%Note that 
%$\gall^1(p)$ transitively acts on $p(\CM)/E$. 
Assume  $T$ is $G$-compact. By the natural embedding,
$G:=\gall^1(p)$ can be considered as a subgroup of $\operatorname{Homeo}(p(\CM)/E)$,
where $p(\CM)/E$ is equipped with the logic topology. Moreover, for
$x:=a/E$ and $G_x:=\{ g\in  G\mid g. x= x\}$  (the stabilizer subgroup of $G$ at $x$, we have that
$G/G_x$ as a homogeneous coset space with the quotient topology  and $p(\CM)/E$ with the logic topology are homeomorphic. Needless to say, analogous facts hold for $\bar p$. 

\item Let $X,Y$ be topological spaces. Recall that if $X$ is  path-connected  then so is any quotient space of $X$.
Moreover, if $X$ is a connected covering space of $Y$, then $X$ is path-connected iff so is $Y$. 

\item Given a  covering map $\delta:X\to Y$, the {\em unique path-lifting property} 
says that for any $x_0\in X$, $y_0\in Y$ with $\delta(x_0)=y_0$, and any  path $\gamma$ in $Y$ starting at $y_0$  (i.e. $\gamma:[0,1]\to Y$ is continuous and $\gamma(0)=y_0$), there is a unique path $\gamma'$ starting at $x_0$ such that $\delta\circ \gamma'=\gamma$. 
\ee
\end{remark}

\begin{proposition}\label{kprescovering}
Assume $T$ is $G$-compact 
 and consider the canonical restriction map $\delta:\p(\CM)/E\to p(\CM)/E$ of  compact spaces $\p(\CM)/E$, $p(\CM)/E$
 equipped with the logic topology. Then $\delta$ is an $n$-covering map. 
 %(i.e., each point
 %in $Gal^1_L(\bar{p})$ has an open neighbourhood which is mapped by $\delta$
 %homeomorphically).
\end{proposition}
\begin{proof} 
 Choose $\phi$ as in the proof of Theorem \ref{$G$-comp} and fix $a'c'\models \p$.
 Define $D:=\p(\CM)\backslash \{a''c''/E:a''c''\models\p\ \wedge
  \neg \phi(a''c'',a'c')\}$. Then $D$ is an open neighborhood of
 $a'c'/E$ in $\p(\CM)/E$. To show that $\delta$ is a covering map, by Remark \ref{covering}, it is enough to see that $\delta$ is injective
 on $D$. So choose two pairs $a^0c^0,a^1c^1 \models \p$ such that
 $a^0c^0/E,a^1c^1/E \in D$ and $\delta(a^0c^0/E)=\delta(a^1c^1/E)$.
 Then, by the first condition, $\models \phi(a^0c^0,a'c')\wedge
 \phi(a^1c^1,a'c')$, and, by the second one, $E(a^0,a^1)$.
 We conclude  by the choice of $\phi$ that
 $a^0c^0/E=a^1c^1/E$. That $\delta$ is $n$-covering is clear due to our choice of $n$. 
\end{proof}

Observe the following remark on covering maps.
 
\begin{remark}\label{coveringhomeo}
 Given topological spaces $X,Y$, suppose that $\delta:X\to Y$ is a $k$-covering. 
 Put $F:=\{f\in \operatorname{Homeo}(X): \delta\circ f=\delta\}$.
 
\begin{enumerate}
 	\item If $Y$ is path-connected, then $|F|\leq k!$ and this bound is optimal.
	\item If $X$ is path-connected (thus so is $Y$), then $|F|\le k$.
\end{enumerate} 
\end{remark}
\begin{proof}
We prove (1) first. It is enough to show that 
 if $f(\in F)$ fixes some fiber $\delta^{-1}(y)$ with $y\in Y$ pointwise \ (*), then $f$ is identity. 
 Fix $x'\in X$ and let $y'=\delta(x')$.
Now there is  a path $\alpha$ from $y$ to $y'$.  Then, by the unique path-lifting property, there is
the unique path $\beta$ starting at $x$, say, and ending at  $x'$ such that $\delta \circ \beta =\alpha$.
Hence, $\delta(x)=y$. Now, since $f\in F$, we have $\delta \circ (f\circ \beta)=\delta \circ \beta =\alpha$ too.
Moreover due to (*),  $f(x)=x$ and by the uniqueness, we have $f\circ \beta=\beta$, so $f(x')=x'$. Therefore $f$ is
the identity map as desired.

To see that the bound is optimal, consider $X=\bigcup_{1\le i\le k} (i-1,i)$ and $Y=(0,1)$. Define a covering map $\delta :X\rightarrow Y, x\mapsto x-\floor{x}$, where $\floor{x}$ is the greatest integer which is less than or equal to $x$. Then, each permutation on $\{1,\ldots, k\}$ induces a homeomorphism of $X$ fixing each fibers of $\delta$ and in this case, $|F|=k!$.

For (2), it is enough to show that if $f\in F$ fixes a point $x$ in $X$, then $f$ is the identity.
Suppose $f\in F$ fixes $x\in X$. If $X$ is path-connected, then for each $x'\in X$, there is a unique 
path $\beta$ from $x$ to $x'$ due to the unique path-lifting property, so
$f$ is the identity by a similar argument as in the proof of (1).
\end{proof}

By   Remarks \ref{factsontop}(2)(3), \ref{coveringhomeo} and Proposition \ref{kprescovering}, 
 we immediately get  
the following:

\begin{corollary} Assume $T$ is $G$-compact.
If $\p(\CM)/E$ or  $p(\CM)/E$ is path-connected (each of which holds  if $\gall^1(p)$ is path-connected, for example  a Lie group), then $|K|=n$. 
\end{corollary}

Now we show a purely compact abelian group theoretical result, which implies Corollary \ref{isom}. 

\begin{lemma}\label{tori}
 Let $G$ be an abelian compact connected topological group, and let $F$ be its finite subgroup. Then $G$ and $G/F$ are isomorphic as topological groups.
\end{lemma}
\begin{proof} Recall that given a finite  group $F$ and a prime number $q$ dividing
  the order of $F$, there is a subgroup of $F$ of order $q$. 
Hence, we can assume that $F$ has a prime order $q$ (applying the prime order case and 
quotienting out finitely many 
times to obtain the conclusion for any finite $F$).

 We can present $G$ as $\lim_{\longleftarrow I}\limits G_i$
with some directed system $(I,\le)$  and continuous homomorphisms $f_{i,j}:G_i\to G_j$ for  $j\leq i\in I$, where
each $G_i$ is an abelian connected Lie group, hence a torus. We 
can assume each $f_{i,j}$ is surjective. 
Since  $F$ has a prime order, it is generated by a single element, say $a=(a_i)_{i\in I}$.
Also, replacing $I$ by $I_{i_0}:=\{i\in I|\ i_0\le i\}$ for an appropriate $i_0$, we can assume that 
$I$ has the least element $i_0$, and that each $F_i$ (i.e., the projection of $F$ onto the $i$-th coordinate) has order $q$. We have that
$$G/F=\lim_{\longleftarrow I} H_i,$$ where $H_i=G_i/F_i$, with maps 
$k_{i,j}:H_i\to H_j$ induced by $f_{i,j}$.  Moreover, for each $i\in I$, we let  $S_i$ be the 
 unique one-dimensional subtorus  of $G_i$
containing $a_i$.
%We may assume that $F$ is a finite cyclic group by the structure theorem of abelian group so that finite abelian group is a finite product of finite cyclic groups.

%\begin{definition}\label{dense_set}
%We say $J\subset I$ is {\em ($\le$-)dense} if for any $i\in I$, there is $j\in J$ such that $i\le j$. 
%\end{definition}

%\begin{fact}\label{properties_dense_set}
%\begin{enumerate}
%	\item For a dense set $J\subset I$, $(J,\le)$ forms a directed system and $\lim_{\longleftarrow I}\limits G_i=\lim_{\longleftarrow J}\limits G_j$.
%	\item Let $I_F:=\{i\in I| |F_i|=|F|\}$. Then $I_F$ is a dense subset of $I$.
%	\item For $i_0\in I$, let $I_{i_0}:=\{i\in I|\ i_0\le i\}$. Then $I_{i_0}$ is a dense subset of $I$.
%\end{enumerate}
%\noindent By Fact \ref{properties_dense_set}, we may assume that
%\begin{itemize}
%	\item There is $i_0\in I$ such that for any $i\in I$, $i_0\le i$. Denote $0:=i_0$.
%	\item For any $i\in I$, $F\cong^{f_i} F_i$.
%\end{itemize}

%\end{fact}

%Let $E_0$ be an open subgroup of $G$, and consider a basis $I=\{E\subseteq E_0 \}$ at the identity of $G$ such that $E\cap F=\{id\}$. Then $(I,\le)$ gives an directed system having the minimum element $0:=E_0$, where $E_i\le E_j$ if $E_i\supset E_j$, and $G_i=G/E_i$. So we may assume that for any $i\in I$, $F_i$ is isomorphic to $F$ by $f_i:G\rightarrow G_i$.

%\begin{theorem}\label{isom_quotient}
%The quotient $G/F$ is isomorphic to $G$. 
%\end{theorem}
%\begin{proof}
\begin{claim}\label{decomp1}
We can find 
%(possibly for a smaller, but equivalent inverse system)
subtori $T_i< G_i$ , $i\in I$, 
such that:
\begin{enumerate}
\item $G_i$ is the direct sum
 of $T_i$ and $S_i$ (in other words, $T_i$ intersects $S_i$
trivially, and $G_i=T_i+S_i$), and
\item $f_{i,j}[T_i]\subseteq T_j$ for each $j\le i$.
\end{enumerate}

\end{claim}
\begin{proof}

For $i=i_0$, we choose any torus $T_{i_0}$ such that $G_{i_0}=T_{i_0}\oplus S_{i_0}$ 
(we can do this just because
$S_{i_0}$ is a subtorus of $G_{i_0}$).

Now, take any other $i\in I$. 
%Put $n=dim(G_{i_0})=dim(S_{i_0})+1$, $m=dim(G_i)$, $k=dim(ker(f_{i,i_0}))$ and 
Put $V_i=f_{i,i_0}^{-1}[T_{i_0}]$. Since both $a_i$ and $a_{i_0}$ have 
order $q$ and $f_{i,i_0}(a_i)=a_{i_0}$, the map $f_{i,i_0}$ maps $S_i$ isomorphically onto
$S_{i_0}$.  
 Hence, $S_i\cap V_i=\{e\}$ (if $x$ was a nontrivial element in 
the intersection, then $f_{i,i_0}(x)$ would be a nontrivial element in $T_{i_0}\cap S_{i_0}$).
Also, for any $x\in G_i$, there are $y\in T_{i_0}$ and $z\in S_{i_0}$ such that
$f_{i,i_0}(x)=y+z$, so choosing $w\in S_{i}$ such that $f_{i,i_0}(w)=z$, we get that
$f_{i,i_0}(x-w)=y\in T_{i_0}$ and $x\in S_i+V_i$. Hence, $G_i$ is the direct sum of
$S_i$ and $V_i$. Note that $V_i$ is a closed (so compact) subgroup of $G_i$. 

%Since $ker(f_{i,i_0})$ intersects  $S_i$ trivially, it must be contained in $V$, so
%$dim(V)=k+(n-1)=m-1$.

Now let $T_i$ be the identity component of the group $V_i$. Then $T_i$ is a torus
of the same dimension as $V_i$  intersecting $S_i$  trivially, so $\dim(T_i+S_i)=\dim(G_i)$. Hence,
by the connectedness of $G_i$, we get that
$G_i=T_i\oplus S_i$. 

To see that $f_{i,j}[T_i]\subseteq T_j$ for  $j\le i$, notice first that
 $f_{i,j}[T_i]\subseteq V_j$.
 %and $f_{i,j}[T_i]\cap S_j=\{id\}$ because $(f_{j,0}\circ f_{i,j})[T_i]=f_{i,0}[T_i]\subset T_0$. 
Now, since $f_{i,j}$ is 
continuous and $T_i$ is connected, $f_{i,j}[T_i]$ is a connected subgroup of 
$G_j$, so it must be contained in the connected component $T_j$ of $V_j$.
This gives the claim.
\end{proof}
Now we will define isomorphisms $g_i$ from $H_i$ to $G_i$ such that, for $j<i$,
$$f_{i,j}g_i=g_jk_{i,j} \ \  (*).$$
(Recall that $k_{i,j}:H_i\to H_j$  is the map induced by $f_{i,j}$.)
Clearly, we can write $H_i=G_i/F_i$ as $T_i\oplus (S_i/\la a_i\ra)$. We define $g_{i}$ to be the identity map on $T_i$, and
to be equal to 
the map $\alpha_q:S_i/\la a_i\ra \to S_i$ induced by multiplication by $q$ of representatives 
modulo
$\la a_i\ra $ on $S_i/\la a_i\ra $, 
and extend additively to a map from $H_i=T_i\oplus (S_i/\la a_i\ra )$ to $T_i \oplus S_i=G_i$.
To check $(*)$, notice that, for $x\in T_i$, both $f_{i,j}g_i(x)$ and $g_jk_{i,j}(x)$ are equal
to $f_{i,j}(x)$, and for $y/\la a_i\ra \in S_i/\la a_i\ra$, we have that  $f_{i,j}g_i(y/\la a_i\ra)$ and $g_jk_{i,j}(y/\la a_i\ra)$ 
are both equal to $qf_{i,j}(y)/\mathbb{Z}$, where $S_j=(\mathbb{R},+)/\mathbb{Z}$.

Now, the system of isomorphisms $g_i$ induces an isomorphism of topological groups $G$ and 
$G/F$.
\end{proof}

\begin{corollary}\label{isom}
Suppose $T$ is  $G$-compact, and
$\gall^1(\bar{p})$ is abelian. Then, $\gall^1(\bar{p})$ and
$\gall^1(p)$ are isomorphic as  topological groups.
\end{corollary}
\begin{proof}
By Proposition \ref{abelian} or Theorem \ref{$G$-comp}, $K$ is finite, hence, by Lemma \ref{tori}, the compact connected  group
$\gall^1(p)$ isomorphic
to $\gall^1(\bar{p})/K$, must be isomorphic to $\gall^1(\bar{p})$ as well. 
\end{proof}

Finally, we observe that Corollary \ref{isom} does not generalize to the non-abelian case.

\begin{proposition}\label{ctrexample}
 Let $G$ be a connected compact Lie group, and $N$ its finite normal subgroup.
 Then, one can find types $p=\tp(a)$ and $\bar p=\tp(ac)$ (with $c\in \acl(a)$ finite) in some $G$-compact theory $T$ 
 with $\acl(\emptyset)=\dcl(\emptyset)$ (in $T^{\eq}$) so that $G=\gall^1(\bar p)$ and $G/N=\gall^1(p)$. 
 The same holds for the groups $\gall^{\res}$ and $\gall^{\lambda}$ for any $\lambda$.
 
  If we  take $G = SO(4,\mathbb{R})$, $N = Z(G) = \{I, -I\}$, then  $G/N = PSO(4,\mathbb{R})$
 and $Z(G/N)$ is trivial, so $G$ and $G/N$ are not isomorphic.\footnote{This example  is pointed out to us
 by Prof. Sang-hyun Kim from Seoul National University.}
\end{proposition}
\begin{proof} The Lie group 
 $G$ is definable in some o-minimal 
expansion $\CR$ of the ordered
field of real numbers, and (using elimination of imaginaries for o-minimal expansions
of $\mathbb{R}$) we identify $\CR$ with $\CR^{eq}$. Put $H:=G/N$.
%By [Ramakrishnan, Peterzil, Eleftheri] $H$ we can assume that $H$ is actually definable in the home sort of
%$\BR$, so there is no confusion as to what ``infinitesimals'' in $H$ mean.
We can assume that $N\subseteq \dcl(\emptyset)=\acl(\emptyset)$, so $H$ is $0$-definable.
We consider a structure $\CM=(\CR,X,Y)$, where  $\CR$ comes with its original structure,  $X$ and $Y$ are sets
equipped with a regular $G$-action and a regular $G/N$-action, respectively (both denoted by $\cdot$), and we 
add to the language the map $\pi:X\to Y$, 
defined as follows: fix any elements $x_0\in X$ and $y_0\in Y$, and, for any $g \in G$, we 
set $\pi(g\cdot x_0):=gN\cdot y_0$. Put $T=\Th(\CM)$.

Note that $\CM=(\CR,X,Y)$ is $0$-interpretable in the structure $(\CR,X)$, so
every automorphism of $(\CR,X)$ extends to an automorphism of $\CM$. Moreover, this extension is unique,
as $Y\subseteq \dcl(X)$. The same holds for the saturated extension $\CM^*=(\CR^*,X^*,Y^*)$, hence, by 
\cite[Chapter 7]{Z},
every automorphism of $\CM^*$ is of the form $\bar{g}\bar{\phi}$, where $g\in G^*$ and $\phi\in \Aut(\CR^*)$
(using the notation from \cite{Z}, and identifying an automorphism of $(\CR^*,X^*)$ with its unique extension to $\CM^*$),
and $\gall(T)=G$.

Let $q$ and $p$ be the unique (strong) types of the sorts $X^*$ and $Y^*$, respectively. 
It follows from \cite{Z} that an automorphism $\bar{g}\bar{\phi}$ is Lascar-strong iff $g\in G^*_{\inf}$ (the group of
infinitesimals of $G^*$).
It follows that two elements $x,y\in X^*$ have the same Lascar type iff there is  $g\in G^*_{\inf}$
such that $g\cdot x=y$, and two elements $w,u\in Y$ have the same Lascar type iff there is $g\in G^*_{\inf}$
and $n\in N$ such that $(gn)\cdot w=u$. As all automorphism of the form $\bar{\phi}$ move any $x\in X^{*}$ to 
its translate by some element of $G^{*}_{\inf}$, we conclude easily that 
$$\autf^1(q)=\Autf^{\res}(q)=\{\bar{g}\bar{\phi}:g \in G^*_{\inf}, \phi\in \Aut(R^*)\}$$ and 
$$\Autf^1(p)=\Autf^{\res}(p)=\{\bar{g}\bar{\phi}:g \in G^*_{\inf}N, 
\phi\in \Aut(R^*)\}.$$
Hence, the map $$g\mapsto [\bar{g}_{|X^*}]$$ is an isomorphism $G\to \gall^1(q)=\gall^{\res}(q)$,
and the map $$g\mapsto [\bar{g}_{|Y^*}]$$ is an epimorphism $G\to \gall^1(p)=\gall^{\res}(p)$ with kernel $N$,
so $\gall^1(p)=\gall^{\res}(p)=G/N=H$.

Now, if we take $c\models q$ and $a=\pi(c)\models p$ and  $\bar{p}=\tp(ac)$,
then $c\in \acl(a)$ (as the fibers of $\pi$ are finite), but  the 
Lascar-Galois groups of $p$ are isomorphic to $H$, and we see  (as $ac$ is interdefinable with $c$) that those of
$\bar{p}$ are isomorphic to $G$.

Finally, since $\gall(T)=G$ is connected, its connected component $\gall^0(T)$ is the same as $G$, so we get by Theorem 21 from \cite{Z} that
all algebraic imaginaries are fixed by all automorphisms over $\emptyset$, hence $\acl(\emptyset)=\dcl(\emptyset)$.
\end{proof}

We will see in Theorem \ref{Th:p_lstp_but_aclp_notlstp} that, in a non-$G$-compact $T$,
some $\tp(\acl(a)/\acl^{\eq}(\emptyset))$ fails to be 
a Lascar type while $\tp(a/\acl^{\eq}(\emptyset))$ is a Lascar type.
We will also see in Proposition \ref{solenoid} that, even in a $G$-compact theory,
$\gall^1(\tp(\acl(a)/\acl^{\eq}(\emptyset))$ may be abelian but  not isomorphic to 
$\gall^1(\tp(a/\acl^{\eq}(\emptyset))$.

%\section{RN but not a Lascar pattern $2$-chain} 
\section{Examples}

As pointed out in the beginning of Section 2, in \cite{DKL}, a motivating result was  observed that in $G$-compact $T$ if $\stp(a)$ is a Lascar type, then so is $\stp(\acl(a))$. In this section we give an example showing that $T$ being $G$-compact is essential in the result. Moreover, 
in regards to  Corollary \ref{isom}, it is natural to ask whether the premise  conditions  are essential. In Example \ref{ctrexample}, we have already seen that the abelianness assumption cannot be removed,
and we present in this section an example showing that the assumption that $c$ is finite cannot be 
removed either.
Lastly, we find a Lie group structure example answering a question raised in \cite{KL}, which asks the existence of an RN-pattern minimal 2-chain  not equivalent to a Lascar pattern 2-chain having the same boundary.    

We now explain preliminary examples. Throughout this section we will use $\CM, \CN,\dots$ to denote some models which need not be saturated.  
For a positive integer $n$, consider a structure $\CM_{1,n}=(M;S,g_n)$, where $M$ is a unit circle; $S$ is a ternary
relation on $M$ such that $S(a,b,c)$ holds iff $a,b,c$ are distinct and $b$ comes before $c$ going around the circle
clockwise starting at $a$; and $g_n=\sigma_{1/n}$, where $\sigma_r$ is the clockwise rotation
by $2\pi r$-radians.

For $p=\tp(a)$, we will denote $\tp(\acl(a))$ by $\bar{\p}$.

\begin{fact}\cite{CLPZ}\label{preli.example}
\be
        \item  $\Th(\CM_{1,n})$  has the unique 1-complete type $p_n(x)$ over $\emptyset$, which is isolated by the formula $x=x$.

        \item $\Th(\CM_{1,n})$ is $\aleph_0$-categorical and has quantifier-elimination.

        \item For any subset $A\subseteq M_n$, $\acl(A)=\dcl(A)=\bigcup_{0\le
        i<n} g_n^i(A)$ (in the home-sort), where $g_n^i=\underbrace{g_n \circ \cdots \circ g_n}_{i~\text{times}}$.
     	\item The unique 1-complete type $p_n$ is also a Lascar type.
\ee
\end{fact}

\subsection{Non $G$-compact theory with $p$ Lascar type but $\bar{\p}$ not Lascar type}
We give an example of non $G$-compact theory which has a Lascar strong type $\tp(a)$ but $\tp(\acl(a))$ is not a 
Lascar strong type. Let $\CM=(M_i,S_i,g_i,\pi_i)_{i\ge 1}$ be a multi-sorted structure where $\CM_{1,i}=(M_i,S_i,g_i)$ (as introduced in Fact \ref{preli.example}, but of course 
$M_i$ and $M_j$ are disjoint for $i\ne j$) 
and $\pi_i:\ M_i\rightarrow M_1$ sending $x\mapsto x^i$ for each
$i\ge 1$ (if we identify each $M_i$ with the unit circle in the complex plane). For each
$n\ge 1$, let $\CM_{\le n}=(M_i,S_i,g_i,\pi_i)_{1\leq i\le n}$. 
Let $T=\Th(\CM)$, $T_i=\Th(\CM_{1,i})$, 
and $T_{\le n}=\Th(\CM_{\le n})$.

\begin{remark}\label{RMK:basic_property_theory}
Both $T$ and $T_{\le n}$'s are $\aleph_0$-categorical.
\end{remark}
\begin{remark}\label{RMK:substr_description}
Let $\CN=(N_i,\ldots)$ be a model of $T$. For $A\subseteq \CN$, define $A_1:=\bigcup_{i\ge 1}\limits \pi_i(\bigcup_{j\le i}\limits g_i^j(A\cap N_i) )$, and $\cl(A)=\bigcup_{i\ge 1}\pi_i^{-1}[A_1].$
\begin{enumerate}
	\item $\cl(A)$ is  the smallest substructure containing $A$.
	\item For $B,C\subseteq \CN$ algebraically closed, if $B_1=C_1$, then $B=C$.
\end{enumerate}

\begin{fact}\cite[Theorem 5.5]{DKL}\label{fact:QE_w_cat}
Let $T$ be $\aleph_0$-categorical and let $\CM=(M,\ldots)$ be a saturated model of $T$. Suppose that if $X\subseteq M^1$ is definable over each of two algebraically closed sets $A_0$ and $A_1$, then $X$ is definable over $B:=A_0\cap A_1$.

% that for all $A$, $\acl(A)=\dcl(A)$, and that for each subset $X$ of $M^1$, if $X$ is $A_0(=\acl(A_0))$-definable and $A_1(=\acl(A_1))$-definable, then $X$ is $B(=A_0\cap A_1)$-definable.

Then, for any subset $Y$ of $M^n$, if $Y$ is both $A_0$-definable and $A_1$-definable, then
it is $B$-definable. Furthermore, in this case, $T$ has weak elimination of imaginaries.
\end{fact}

\end{remark}
\begin{proposition}\label{Prop:T_QE_EI}
\begin{enumerate}
	\item Both $T$ and $T_{\le n}$, $n\ge 1$, have quantifier elimination.
	\item Both $T$ and $T_{\le n}$, $n\ge 1$, weakly eliminate imaginaries.
\end{enumerate} 
\end{proposition}
\begin{proof}
$(1)$ Here we prove that $T$ has quantifier elimination. By a similar reason, each $T_{\le n}$ has also quantifier elimination. Let $\CN_1=(N_{i}^1,\ldots)$ and $\CN_2=(N_{i}^2,\ldots)$ be a model of $T$. Take finite $A\subseteq \CN_1$, and let $f: \cl(A) \rightarrow \CN_2$ be a partial embedding. Take $a\in \CN_1$ arbitrary. It is enough to show that there is a partial embedding $g :\cl(Aa)\rightarrow \CN_2$ extending $f$. By Remark \ref{RMK:substr_description}(1), we may assume that $A=A_1\subset N_1^1$. Suppose $a\in N_i^1$. Let $a_1=\pi_i(a)$. If $a_1\in A$, then $a\in \acl(A)$, and we are done. 
Without loss of generality, we may assume that $a_1\not\in A$. Since $T_1$ has quantifier elimination, we can pick $b_1\in N_{1}^2$ such that $b_1\models \tp_{T_1}(a_1/A)$. Let $A_i=\pi_i^{-1}[A]\subseteq N_i^1$, which is finite, and let $\a_i=(a_i^1,\ldots,a_i^i)$ be an enumeration of $\pi_i^{-1}(a_1)$. Since each $T_i$ has quantifier elimination, we can pick an enumeration $\b_i=(b_i^1,\ldots,b_i^i)$ of $\pi_i^{-1}(b_1)$ such that $\b_i\models \tp_{T_i}(\a_i/A_i)$. Note that $\cl(Aa)=\cl(A)\cup \bigcup_{i\ge 1}\limits \pi_i^{-1}(a_1)$. Consider a map $g=f\cup\{(a_i^j,b_i^j)|\ 1\le i,1\le j\le i\}$. For each $i$, $g\restriction_{N_i^1}$ is a
partial embedding to $N_i^2$, and for each $j\le i$, $b_1=\pi_i(b_i^j)=\pi_i(g(a_i^j))$. 
Therefore $g:\cl(Aa)\rightarrow \CN_2$ is a partial embedding extending $f$, and we are done. As a consequence, $\acl(A)=\cl(A)$ for each $A$.

$(2)$ We first show that each $T_{\le n}$ weakly eliminates imaginaries. We consider $\CM_{\le n}$ as $\CM_{\le n}'=(M',S',g',\pi')$, where $M'=M_1\cup\cdots \cup M_n$, $ S'=S_1\cup\cdots\cup S_n$, $g'=g_1\cup\cdots\cup g_n$, and $\pi'=\pi_1\cup\cdots\cup \pi_n$. Let $T_{\le n}'=\Th(\CM_{\le n}')$. It is enough to show that $T_{\le n}'$ weakly eliminates imaginaries. Note that $T_{\le n}'$ is $\aleph_0$-categorical. Let $\CN=(N(=N_1\cup\cdots\cup N_n),\ldots)$ be a saturated model of $T_{\le n}'$. By Fact \ref{fact:QE_w_cat},
it is enough to show that a subset $X$ of $N^1$ which is both $A$-definable and $B$-definable, where $A$ and $B$ are algebraically closed, is also definable over $C=A\cap B$. Let $A_i$, $B_i$ and $C_i$ denote the intersections of $A$, $B$, and $C$, respectively, with $N_i$, $1\le i \le n$. Note that $C_i=A_i\cap B_i$ for each $1\le i\le n$ and $C=\acl(C_1)$. We may assume that $X\subseteq N_{i}$ for some $1\le i\le n$ by taking $X\cap N_{i}$. By quantifier elimination, $X$ is definable over $A_{i}$ and $B_{i}$. Then from the proof of \cite[Theorem 4.3(1)]{KKL}, $X$ is definable over $C_{i}=A_{i}\cap B_{i}$ and clearly over $C=\acl(C_i)$. Thus $T_{\le n}'$ weakly eliminates imaginaries and so does $T_{\le n}$.

Let $\CM'$ be a saturated model of $T$ and let $\CM_{\le n}'$ be the saturated structure corresponding to $\CM_{\le n}$. By Remark \ref{RMK:substr_description}(1) and quantifier elimination, there is no strictly decreasing chain of algebraically closures of finite sets in $\CM'$. Since $\CM'=\bigcup_n \CM_{\le n}$, by the same reasoning as 
in the proof of \cite[Theorem 5.9]{DKL}, we conclude that $T$ weakly eliminates imaginaries. 
\end{proof}
Let $\CN=(N_i,\ldots)$ be a saturated model of $T$. By Proposition \ref{Prop:T_QE_EI},
$\acl^{\eq}(\emptyset)=\emptyset$. Let $a\in N_1^1$, and consider strong types
$p:=\tp(a)$ (isolated by $x=x$) and   $\bar{\p}:=\tp(\acl(a))$.

\begin{fact}\cite{CLPZ}\label{fact:Lascar_distance_M_n}
For $a\neq b\in N_n$, $a$ and $b$ have the same type over some elementary
substructure of $N_n$ if and only if $N_n\models S_n(a,b,g_n(a))\wedge S_n(g_n^{-1}(a),b,a)$. So there are $a,b\in N_n$ whose Lascar distance is at least $n/2$. 
\end{fact}

\begin{theorem}\label{Th:p_lstp_but_aclp_notlstp}
The type $p$ is a Lascar strong type but $\bar{\p}$ is not a Lascar strong type.
\end{theorem}
\begin{proof}
By Fact \ref{fact:Lascar_distance_M_n}, $p$ is a Lascar strong type. Take $a,b\in N_1$ which are Lascar 
equivalent. Let $\a_i=(a_i^1,\ldots,a_i^i)$ and $\b_i=(b_i^1,\ldots,b_i^i)$ be enumerations of
$\pi_i^{-1}(a)$ and $\pi_i^{-1}(b)$. Denote $\sigma_i(a_i)=(a_i^{\sigma(1)},\ldots,a_i^{\sigma(i)})$
for each permutation $\sigma\in S_i$. Let $\tau_i$ be the permutation of $\{1,2,\ldots, i\}$ sending
$j$ to $j+1$ for $j< i$ and $i$ to $1$. 
Then $(a_1,\tau_2^{k_2}(\a_2),\tau_3^{k_3}(\a_3),\ldots)\models \p$ for arbitrary
$k_2,k_3,\ldots$ $(*)$. For each $i$, there are $a_i\in \pi_i^{-1}(a_1)$ and $b_i\in \pi_i^{-1}(b_1)$ 
whose Lascar distance is at least $n/2$. So, by $(*)$, $\bar{\p}$ is not a Lascar strong type.  
\end{proof}

\subsection{A $G$-compact theory with $\gall^1(\bar{\bar{p}})$ abelian and not isomorphic to $\gall^1(p)$}
We will use the symbols $\sigma_r$, $g_n$ and $(M_i, S_i)$ defined in the previous subsection.
Put $\tilde{\CM}=\tilde{\CM}_1=(M_1,S_1,g_{1,n})_{n\geq1}$, where $g_{1,n}$ is the same rotation as $g_n$ on the unit circle $M_1=S^1$. 
We define a multi-sorted structure  $$\CM'=((\tilde{\CM}_k),\delta_k)_{k\geq 1},$$ where each $\tilde{\CM}_k:=
(M_k=\{x_{(k)}:x \in S^1\},S_k,g_{k,n})_{n\geq 1}$ is a copy of
$\tilde{\CM}_1$ (again $g_{k,n}$ is the same rotation as $g_n$ on $M_k$, and we may omit $k$ in $S_k$ and $g_{k,n}$), 
and $\delta_k:\tilde{\CM}_{k+1}\to \tilde{\CM}_k$ is given by 
$\delta_k (x_{(k+1)})=x^2_{(k)}$.
Put $T'=\Th(\CM')$.
%\begin{remark}\label{extending}
 %If $f:\tilde{\CM}_{\leq k}\to\tilde{\CM}_{\leq k}$ is a partial automorphism, we can extend it
 %to a partial automorphism $\tilde{f}:\tilde{\CM}_{\leq k+1}\to\tilde{\CM}_{\leq k+1}$
 %with $dom(\tilde{f})=dom(f)\cup\pi_k^{-1}[dom(f)\cap \tilde{\CM}_k]$ by putting $\tilde{f}(x_{(k+1)})=\sigma_{r/2}(x)_{(k+1)}$,
 %here $r$ is such that for $y_{(k)}=\pi(x_{(k+1)})$, $f(y_{(k)})=\sigma_r(y)_{(k)}$ (so $\tilde{f}$ is induced
  %by $f$ on each subsemicircle of $\tilde{\CM}_{k+1}$ by identifying the semicircle with $\tilde{\CM}_{k}$
  %via $\pi_k$).
%\end{remark}

\begin{lemma}\label{qe_m'}
 $T'$ admits q.e.
\end{lemma}
\begin{proof}
It is enough to prove q.e. for the  restriction of $\CM'$ to finitely many sorts
$\tilde{\CM}_{\leq k}$ for each $k$,
but such a restriction is quantifier-free interpretable in the structure $\CM'$ for some $n$, which is known to admit q.e. 
(\cite[Theorem 5.8]{DKL}).
\end{proof}

Using q.e, it is easy to check that two (possibly infinite) tuples have the same Lascar type iff their corresponding coordinates are
infinitesimally close (in the same sort), and $S$ induces corresponding partial orders on the tuples.
As this is a type-definable condition, we get that $T'$ is $G$-compact.

By Theorem 5.9 from \cite{DKL}, $\tilde{\CM}$, and hence also any $\tilde{\CM}_{\leq k}$, weakly eliminates
imaginaries. Hence, we get:
\begin{remark}\label{wei}
$T'$ weakly eliminates imaginaries.
\end{remark}
By Lemma \ref{qe_m'}, one easily gets that $\acl_{T'}(\emptyset)=\emptyset$ in the home-sort, hence, by Remark \ref{wei}, 
$\acl_{T'}^{\eq}(\emptyset)=\emptyset$.
Now, let $p$ be the (unique) type of an element $a\in \tilde{\CM}$. Then 
$\gall^1(p)\cong \gall(\Th(\tilde{\CM}))\cong S^1$. On the other hand, for $\bar{\p}=\tp(\acl(a))$, we have:

\begin{proposition}\label{solenoid} 
  $\gall^{1}(\bar{\bar{p}})$ is isomorphic to the $2$-solenoid 
  $G:=\varprojlim_i \mathbb{R}/2^i\mathbb{Z}$
  (the projections in the inverse limit are the
natural quotient maps). Hence, $\gall^{1}(\bar{\bar{p}})$ is abelian and non-isomorphic to 
$\gall^{1}(p)$.
\end{proposition}
\begin{proof}
 Consider the homomorphism
$\phi:G=\varprojlim_i \mathbb{R}/2^i\mathbb{Z}\to \gall^{1}(\bar{\bar{p}})$ 
sending a sequence $(r_i/2^i\mathbb{Z})_i$ 
to the class induced by the automorphism $f_{(x_i)_i}$ (defined on a monster model of $T'$) equal
to $\sigma_{r_i/2^i}$ on $\tilde{\CM}_{i}$ for each $i$. If $r_i\notin 2^i\mathbb{Z}$, then $f_{(r_i)_i}$ does not preserve the Lascar types of
elements of $\tilde{\CM}_{i}$, hence $\ker(\phi)=\{0\}$. It remains to check that $\phi$ is surjective.
Consider a class in $\gall^{1}(\bar{\bar{p}})$ induced by an automorhpism $f$. As all elements
of $ \tilde{\CM}_{1}$ have the same Lascar type, we  may assume that
$f(a)=a$ for some $a\in \tilde{\CM}_{1}$ (by composing $f$ with a strong automorphism sending $f(a)$ to $a$).
Now, for each $i$, as $f$ commutes with $\delta_1\delta_2\dots\delta_{i-1}$, it must preserve $\acl(a)\cap \tilde{\CM}_{i}$ setwise.
Hence,  there is $r_i$ such that on $\acl(a)\cap \tilde{\CM}_{i}$, $f$ is equal to 
$\sigma_{r_i/2^i}$. Then $f$ and $f_{(r_i)_i}$ have infinitesimally close values on every element of the monster model,
so $f^{-1}f_{(r_i)_i}$ is Lascar strong and $\phi(f_{(r_i)_i})$ is equal to the class induced by $f$.
\end{proof}

\subsection{RN but not a Lascar pattern $2$-chain}
%[State the definitions of  RN-pattern and Lascar pattern $2$-chains with the $1$-shell boundaries. We change terminology RN, NR, Lascar ..-{\em types} in \cite{KKL},\cite{KL} to those  {\em patterns} in order not to confuse with the existing notion of types of tuples.]
%\medskip 

Let $p$ be a strong type of an algebraically closed set. In \cite{KKL}, it was shown that any $2$-chain in $p$ with 
$1$-shell boundary can be reduced to a $2$-chain having an  
{\em RN-pattern} or an {\em NR-pattern}.\footnote{In this paper,
we change terminology ``$\dots$-type'' to ``$\dots$-pattern'' in order to avoid confusion
with the existing notion of a type - a set  of formulas.} Also for $a,b\models p$, if $a\equiv^L b$, then $(a,b)$ is an endpoint pair of a $1$-shell which is the boundary of a $2$-chain having a {\em Lascar pattern}, a kind of an
RN-pattern. In \cite[Question 4.5]{KL}, it was asked whether all minimal $2$-chains having an RN-pattern with $1$-shell boundary should be equivalent to that having a Lascar pattern. It was also noticed that if {\color{red}there is a counterexample}, then its length should be at least $7$. 

In this subsection, we assume the reader has some familiarity with the  notions described in \cite{KKL, KL},
and we  work with trivial independence to define independent functors in a strong type (see Definition \ref{p-*-functors}), and corresponding $1$-shells and $2$-chains. We now give descriptions of $2$-chains having an RN- or a Lascar pattern in terms of balanced walks, rather than original definitions. For more combinatorial descriptions for RN- or Lascar pattern $2$-chains, see \cite{KL}.

\begin{remark/def}\label{rem/def:balanced chain walk}
Let $a,b\models p$. For $n\ge 1$, a {\em balanced edge-walk of length $2n$ from $a$ to $b$} is a finite sequence $(d_i)_{0\le i\le 2n}$ of realizations of $p$ satisfying the following conditions:
\be
	\item $d_0=a$, and $d_{2n}=b$; and
%	\item $\{d_{j},d_{j+1}\}$ is $*$-independent for each $j\le 2n+1$; and
	\item there is a bijection $$\sigma:\{0,1,\ldots,n-1\}\to\{0,1,\ldots,n-1\}$$ such that
	$d_{2i} d_{2i+1}\equiv d_{2\sigma(i)+2}d_{2\sigma(i)+1}$ for $ i < n$),
\ee
A balanced edge-walk of length $2n$ from $a$ to $b$  induces a $2$-chain of length $2n+1$ (having the $1$-shell boundary whose end point pair is $(a,b)$), which is a chain-walk (c.f. the proof of \cite[Theorem 4.2]{DKL}). Such an induced $2$-chain need not be unique but it is unique up to the first homology class. 

 In particular, given a minimal  $2$-chain $\alpha$ of length $2n+1$ with the $1$-shell boundary $s$,  $\alpha$ is equivalent to a Lascar pattern $2$-chain   iff   $\alpha$ is equivalent to 
 a chain-walk induced by a balanced  edge-walk of length $2n$ from $a$ to $b$ with $\sigma=\id$ such that $[a,b]=[s]\in H_1(p)$.
\end{remark/def}

\begin{Fact}\label{rnpattern}\cite{DKL, KL}
Let $\alpha$ be a {\em minimal} $2$-chain of length $2n+1$ ($n\geq 1$) in $p$ with the
boundary $s=g_{12} -g_{02} + g_{01}$ such that  $\supp(g_{ij})=\{i,j\}$. 
%Let $(a,b)$ be an endpoint pair of $s$.
   Then the following are equivalent.
   \be\item The $2$-chain
    $\alpha$  has an RN-pattern.
   
    \item The $2$-chain $\alpha$ is equivalent to a $2$-chain
    $$\alpha'=\sum_{i=0}^{2n}\limits (-1)^i f_i$$
     which is a chain-walk of $2$-simplices  from $g_{01}$ to $-g_{02}$ such that $\Bd^2 f_{0}=g_{01}$, $\Bd^0 f_{2n}=g_{12}$, $\Bd^1(f_{2n})=-g_{02}$ and  $\supp(\alpha')=3=\{0,1,2\}$. 
    
    \item
    There are an endpoint pair $(a,b)$ of $s$, and
    a balanced-edge-walk of length $2n$ from $a$ to $b$ which induces a $2$-chain equivalent to $\alpha$.
\ee 
  % (The representation of $\alpha'$ is called {\em standard}.)
  \end{Fact}

%Note that  if $(a,c,d,b), (a,c,d',b')$ are  representations of $s$ (so that $(a,b')$ is another endpoint pair of $s$), then 
%since $cd\equiv cd'$ and $db\equiv d'b'$, one may write a more general description of   Fact \ref{rnpattern}(3) for an arbitrary  %endpoint pair of $s$.

%\medskip
%\subsection{Some examples of RN but not Lascar pattern $2$-chains}
In this subsection we will give a minimal $2$-chain example (with 1-shell boundary)  having an  RN-pattern but not equivalent to that having a Lascar pattern. The example will have length $7$. We start with preparatory constructions.
Let $G$ be a connected compact Lie group definable in $\BR$ in the ordered ring language. Let $X$ be a set equipped with a regular $G$-action. Let $\CM_2=(\CR,X,\cdot)$ be a two sorted structure, where $\CR$ is 
the field of real numbers with a named  finite set of elements over which $G$ is definable. Let $\CM_n=(\CM_{1,n}, \CM_2)$, and there is no interaction between $\CM_{1,n}$ and $\CM_2$.

Let $\CN_n^*=(\CM_{1,n}^*,\CM_2^*)$ be a saturated model of $\Th(\CM_n)$ such that $\CM_{1,n}^*=(M^*,S,g_n)$ and $\CM_2^*=(\CR^*,X^*,\cdot)$ are saturated models of $\Th(\CM_{1,n})$ and $\Th(\CM_2)$ respectively. Since there are no other interactions between $\CM_{1,n}^*$ and $\CM_2^*$, 
we have that
\begin{itemize}
	\item $\aut(\CN_n^*)=\aut(\CM_{1,n}^*)\times \aut(\CM_2^*)$.
	\item For $a\in \CM_{1,n}^*$ and $b\in \CM_2^*$, $\acl(ab)=\acl_{\CM_{1,n}}(a)\cup \acl_{\CM_2}(b)$, where $\acl_{\CM_{1,n}}$ and $\acl_{\CM_2}$ are the algebraic closures taken in $\CM_{1,n}^*$ and $\CM_2^*$ respectively.
	\item For $a\in \CM_{1,n}^*$ and $b\in \CM_2^*$, $\tp(ab)=\tp_{\CM_{1,n}}(a)\cup \tp_{\CM_2}(b)$ where $\tp_{\CM_{1,n}}(a)$ and $\tp_{\CM_2}(b)$ are complete types in $\CM_{1,n}^*$ and $\CM_2^*$ respectively.
\end{itemize} 

First, let us recall a fact on minimal lengths of $1$-shells in $p_n$ from \cite{KKL}.
\begin{definition}
Let $\CA$ be a non-trivial amenable collection and let $s$ be a $1$-shell. Define $$B(s) :=\min \{~|\tau|  :  \tau~\text{is a (minimal) 2-chain and}~\partial(\tau)=s ~\}.$$ If $s$ is not the boundary of any $2$-chain,  define $B(s) := - \infty$.
\end{definition}

\begin{fact}\label{lengthbound}
For $n\ge7$, there is a $1$-shell $s$ in $p_n$ such that $B(s)=7$.   
\end{fact}
\noindent Note that in \cite{KKL}, the authors considered $1$-shells and $2$-chains defined under 
$\acl$-independence for Fact \ref{lengthbound}. By similar methods as in the proofs, it is not hard to see that the same Fact \ref{lengthbound} holds for $1$-shells and $2$-chains defined under trivial independence as in this subsection.

Next, let us recall some facts about $\CM_2^*$ from \cite{Z}. Let $G^*$ be the extension of $G$ in $\CR^*$. Let $\mu$ be the normal subgroup of infinitesimals in $G^*$. We fix a base point $x_0\in X^*$ so that for $x\in X^*$ there is a unique element $h\in G^*$ with $x=h\cdot x_0$.
\begin{fact}\label{facts_auto_(R,X)}
\begin{enumerate}
	\item Each $\phi\in \aut(\CR^*)$ is extended to $\aut(\CM^*)$ defined by $\bar{\phi}(h\cdot x_0):=\phi(h)\cdot x_0$. An automorphism which fixes $\CR^*$ pointwise is of the form $\bar{g}$ for some $g\in G^*$, where $\bar{g}(h\cdot x_0)=(hg^{-1})\cdot x_0$, and the commutation rule is given by $\bar{\phi}\bar{g}=\ov{\phi(g)}\bar{\phi}$. So we have $\aut(\CM^*)=\aut(\CR^*)\ltimes G^*$.
	\item An automorphism $\Phi=\bar{\phi}\bar{g}$ is a strong automorphism if and only if $g$ is an infinitesimal. So the Lascar Galois group of $\Th(\CM)$ is isomorphic to $G$.
	\item $\mu=\{h^{-1}\phi(h)|\ h\in G^*,\ \phi\in \aut(\CR^*)\}$.
\end{enumerate}
\end{fact}

\noindent We note that in $X^*$, there is a unique $1$-type over $\emptyset$ and it is also a strong type over $\emptyset$.
\begin{proposition}\label{stp_in_X}
For any $x\in X$, $x\equiv x_0$ and $\tp(x_0)=\stp(x_0)$.
\end{proposition}
\begin{proof}
Suppose there are $x,x'\in X^*$ with $\stp(x)\neq\stp(x')$. Then there is a $\emptyset$-definable finite equivalence relation $E$ on $X^*$ such that $\neg E(x,x')$. Define $H:=\{g\in G^*|\ E(x,g\cdot x)\}$. Since $E$ is $\emptyset$-definable, $H$ is $\emptyset$-definable subgroup of $G^*$, 
and it is proper. We claim that $[G^*:H]$ is finite. Note that $gg^{-1}\in H$ if and only if $E(g\cdot x,g'\cdot x)$. So the map sending $gH\in G^*/H\mapsto g(x)E\in X^*/E$ is a bijection because $G^*$ acts on $X^*$ transitively. Since $X^*/E$ is finite, so is $G^*/H$, contradicting the connectedness of $G$.
\end{proof}

\noindent Let $\acl_{\CR^*}^{\eq}(\emptyset)$ be the algebraic closure of $\emptyset$ in $(\CR^*)^{\eq}$. For $x\in X^*$, $\acl^{\eq}(x)=\{x\}\cup\acl_{\CR^*}^{\eq}(\emptyset)$. From Proposition \ref{stp_in_X}, we have that $\tp(x)\models \stp(\acl^{\eq}(x))$. So we say $x_1$ and $x_2$ are endpoints if $\acl^{\eq}(x_1)$ and $\acl^{\eq}(x_2)$ are endpoints of a $1$-shell in $p:=\stp(\acl^{\eq}(x))$.

\begin{lemma}\label{infinitesimal_1step_lascar}
For $x_1, x_2, x_3\in X^*$, if $x_1x_2\equiv x_3x_2$, then there is $h\in \mu$ such that $\bar{h}(x_1)=x_3$.
\end{lemma}
\begin{proof}
Let $x_1,x_2,x_3\in X^*$ be such that $x_1x_2\equiv x_3x_2$. Let $h_1,h_2,h_3\in G^*$ be such that $x_i=h_i\cdot x_0$. It is enough to show that $h_3h_1^{-1}\in\mu$.

Take $\phi \in \aut(\CR^*)$ and $g\in G^*$ such that $\bar g \bar \phi(x_1x_2)=x_3x_2$. Then we have
\begin{align*}
\bar g \bar \phi(x_1x_2) &=\bar g((\phi(h_1)\cdot x_0) (\phi(h_2) \cdot x_0))\\
&=(\phi(h_1)g^{-1}\cdot x_0)(\phi(h_2)g^{-1}\cdot x_0)\\
&=(h_3\cdot x_0)(h_2\cdot x_0). 
\end{align*}
By the regularity of the action of $G^*$, we have $\phi(h_1)g^{-1}=h_3$ and $\phi(h_2)g^{-1}=h_2$, and so $g=h_2^{-1}\phi(h_2)\in \mu$. And $h_3h_1^{-1}=\phi(h_1)g^{-1}h_1^{-1}=\phi(h_1)h_1^{-1}=\id$ modulo $\mu$ because $\mu$ is a normal subgroup of $G^*$. Thus $h_3h_1^{-1}\in \mu$.
\end{proof}

\begin{theorem}\label{classification_lascar_type}
For $x_1,x_2\in X^*$, the following are equivalent:
\begin{enumerate}
	\item $x_1\equiv^{L}x_2$.
	\item There is $h\in \mu$ such that $\ov h(x_1)=x_2$.
	\item $x_1$ and $x_2$ are endpoints of a $1$-shell which is the boundary of a $2$-chain in $p$ having a Lascar pattern.
\end{enumerate}
\end{theorem}
\begin{proof}
$(1) \Leftrightarrow (2)$ was proved in \cite{Z} and $(1)\Rightarrow(3)$ was proved 
in \cite{KKL}. It is enough to show $(3)\Rightarrow(2)$. By definition of Lascar pattern
 2-chain (in \cite{KKL}) and Lemma \ref{infinitesimal_1step_lascar}, $(3)\Rightarrow(2)$ is also true.
\end{proof}

Now we give a promised example of a minimal $2$-chain having an RN-pattern but 
not equivalent to one having a Lascar pattern. Denote $[g,h]=ghg^{-1}h^{-1}$. Consider $\CN_n^*$ for some $n\ge 7$ and let $q=\tp(ab)$ for $a\models p$ and $b\models p_n$, which is a strong type.
\begin{theorem}\label{RN_but_not_Lascar}
Suppose $G$ is not abelian. Then there is a minimal $2$-chain in $q$ which has an RN-pattern but not equivalent to 
 any Lascar pattern 2-chain.
\end{theorem}
\begin{proof}
Suppose $G$ is not abelian. We can choose $g_1, g_2, g_3\in G$ such that $[(g_3g_2)^{-1},(g_2g_1))]\neq \id$. Let $h_1 = g_1$, $h_2=g_2$, $h_3=g_3$, $h_4=h^{1-}[h_1,(h_3h_2)]$, $h_5=h^{-2}[h_2,(h_4h_3)]$, $h_6=h^{-3}[h_3,(h_5h_4)]$, and $f_i=h_i h_{i-1}\ldots h_1$ for $1\le i\le 6$. Set $x_i=(h_i h_{i-1}\cdots h_1)^{-1}\cdot x_0$ for $1\le i\le 6$. Consider a $1$-shell $s=f_{01}+f_{12}-f_{02}$, where $f_{01}(\{01\})\equiv [x_0 x_0]$, $s=f_{12}(\{1,2\})=[x_0 x_0]$, and $f_{02}(\{0,2\})\equiv [x_6 x_0]$, and its endpoints are $x_0$ and $x_6$. Set $h_1:=g_1$, $h_2:=g_2$, $h_3:=g_3$, $h_4:=h_1^{-1}[h_1,(h_3h_2)]$, $h_5:=h_2^{-1}[h_2,(h_4h_3)]$, and $h_6:=h_3^{-1}[h_3,(h_5h_4)]$. Let $f_i=h_ih_{i-1}\cdots h_1\in G\subseteq G^*$ for $1\le i\le 6$. Then we have that 
\begin{itemize}
	\item $f_4=f_3f_1^{-1}$;
	\item $f_5f_1^{-1}=f_4f_2^{-1}$; and
	\item $f_6f_2^{-1}=f_5f_3^{-1}$.
\end{itemize}
\noindent This implies $\bar f_4(x_0x_1)=x_4x_3$, $\ov{f_5f_1^{-1}}(x_1x_2)=x_5x_4$, and $\ov {f_6f_2^{-1}}(x_2x_3)=x_6x_5$. Then we have that $x_0x_1\equiv x_4x_3$, $x_1x_2\equiv x_6x_5$, and $x_2x_3\equiv x_5x_4$.
$$\begin{tikzcd}
x_0 \ar[r, "\bar{h}_1"]&x_1 \ar[r, "\bar{h}_2"]& x_2 \ar[r, "\bar{h}_3"]& x_3 \ar[r, "\bar{h}_4"]& x_4 \ar[r, "\bar{h}_5"]& x_5 \ar[r, "\bar{h}_6"]& x_6
\end{tikzcd}$$

\noindent This gives a $2$-chain $\alpha$ of length $7$ having the  $1$-shell boundary $s$. Note that
\begin{align*}
f_6&= (h_6h_5h_4h_3)h_2h_1\\
&=(h_5h_4)h_2h_1\\
&=(h_4h_3h_2^{-1}h_3^{-1})h_2h_1\\
&=(h_4h_3)(h_2^{-1}h_3^{-1})(h_2h_1)\\
&=(h_3h_2)(h_2h_1)^{-1}(h_2^{-1}h_3^{-1})(h_2h_1)\\
&=[(h_3h_2),(h_2h_1)^{-1}]\neq \id,
\end{align*}
and $f_6\notin \mu$. By Theorem \ref{classification_lascar_type}, $s$ is not the boundary of $2$-chain with a Lascar pattern and $B(s)\le 7$.\\

Let $b,b'\models p_n$ where $(b,b')$ is an endpoint pair of a $1$-shell $s'$ of $p_n$ such that $B(s')=7$ in $\CM_{1,n}$. Let $\alpha'$ be a minimal $2$-chain with an RN-pattern having the $1$-shell boundary $s'$ of length $7$. Note that we can take such a $2$-chain because $p_n$ is a Lascar type.
Consider $a=(x_0,b),a'=(x_6,b')\models q$. Let $s''$ be a $1$-shell induced from $s$ and $s'$ with an endpoint pair $(a,a')$. Consider a $2$-chain $\alpha''$ in $q$ induced from $\alpha$ and $\alpha'$, which
does not  a  Lascar pattern and has length $7$. We claim that $\alpha''$ is minimal. Since $B(s)\le 7$ and $B(s')=7$, we have that $B(s'')=7$ and so $\alpha''$ is minimal. By the construction of $\alpha''$, it has an RN-pattern. Therefore, the minimal $2$-chain $\alpha''$ has an RN-pattern but is not equivalent to a  Lascar pattern $2$-chain.
\end{proof}

%[State more precise examples written by Junguk and Jan.]

\end{document}